\documentclass[11pt]{amsart}
\usepackage{amscd}
\usepackage{amsmath}
\usepackage{amsxtra}
\usepackage{amsfonts}
\usepackage{amssymb}
\usepackage{color}

\newtheorem{theorem}{Theorem}[section]

\newtheorem{lemma}[theorem]{Lemma}
\newtheorem{proposition}[theorem]{Proposition}

\newtheorem{conjecture}[theorem]{Conjecture}
\theoremstyle{definition}
\newtheorem{definition}[theorem]{Definition}
\newtheorem{remark}[theorem]{Remark}

\newtheorem{example}[theorem]{Example}
\theoremstyle{remark}

\newcommand{\eproof}{\hfill$\square$}
\renewcommand{\theclaim}{\textup{\theclaim}}

\numberwithin{equation}{section}

\def\openone

{\mathchoice

{\hbox{\upshape \small1\kern-3.3pt\normalsize1}}

{\hbox{\upshape \small1\kern-3.3pt\normalsize1}}

{\hbox{\upshape \tiny1\kern-2.3pt\SMALL1}}

{\hbox{\upshape \Tiny1\kern-2pt\tiny1}}}

\makeatletter

\newbox\ipbox

\newcommand{\diracb}[1]{\left\langle #1\mathrel{\mathchoice

{\setbox\ipbox=\hbox{$\displaystyle \left\langle\mathstrut
#1\right.$}

\vrule height\ht\ipbox width0.25pt depth\dp\ipbox}

{\setbox\ipbox=\hbox{$\textstyle \left\langle\mathstrut
#1\right.$}

\vrule height\ht\ipbox width0.25pt depth\dp\ipbox}

{\setbox\ipbox=\hbox{$\scriptstyle \left\langle\mathstrut
#1\right.$}

\vrule height\ht\ipbox width0.25pt depth\dp\ipbox}

{\setbox\ipbox=\hbox{$\scriptscriptstyle \left\langle\mathstrut
#1\right.$}

\vrule height\ht\ipbox width0.25pt depth\dp\ipbox}

}\right. }

\newcommand{\dirack}[1]{\left. \mathrel{\mathchoice

{\setbox\ipbox=\hbox{$\displaystyle \left.\mathstrut
#1\right\rangle$}

\vrule height\ht\ipbox width0.25pt depth\dp\ipbox}

{\setbox\ipbox=\hbox{$\textstyle \left.\mathstrut
#1\right\rangle$}

\vrule height\ht\ipbox width0.25pt depth\dp\ipbox}

{\setbox\ipbox=\hbox{$\scriptstyle \left.\mathstrut
#1\right\rangle$}

\vrule height\ht\ipbox width0.25pt depth\dp\ipbox}

{\setbox\ipbox=\hbox{$\scriptscriptstyle \left.\mathstrut
#1\right\rangle$}

\vrule height\ht\ipbox width0.25pt depth\dp\ipbox}

} #1\right\rangle}

\newcommand{\beq}{\begin{equation}}
\newcommand{\eeq}{\end{equation}}

\def\blfootnote{\xdef\@thefnmark{}\@footnotetext}

\renewcommand{\mod}{\operatorname{mod}}

\input xy
\xyoption{all}
\usepackage{amssymb}

\def\R{\mathbb{R}}

\def\-{^{-1}}

\def\Z{\mathbb{Z}}

\begin{document}

\title{On Spectral Cantor-Moran measures and  a variant of Bourgain's sum of sine problem}
\author{Lixiang An}
\address{[Lixiang An]School of Mathematics and Statistics,
$\&$ Hubei Key Laboratory of Mathematical Sciences,
Central China Normal University,
Wuhan 430079,
P.R. China.}

 \email{anlixianghai@163.com}

\author{Xiaoye Fu}
\address{[Xiaoye Fu]School of Mathematics and Statistics,
$\&$ Hubei Key Laboratory of Mathematical Sciences,
Central China Normal University,
Wuhan 430079,
P.R. China.}

 \email{xiaoyefu@mail.ccnu.edu.cn}

\author{Chun-Kit Lai}

\address{[Chun-Kit Lai]Department of Mathematics, San Francisco State University,
1600 Holloway Avenue, San Francisco, CA 94132.}

 \email{cklai@sfsu.edu}

\thanks{The research of Lixiang An and Xiaoye Fu is supported by NSFC  grant 11601175, 11401205.}

\begin{abstract}
In this paper, we show that if we have a sequence of Hadamard triples $\{(N_n,B_n,L_n)\}$ with $B_n\subset \{0,1,..,N_n-1\}$ for $n=1,2,...$, except an extreme case, then the associated Cantor-Moran measure 
$$
\begin{aligned}
\mu = \mu(N_n,B_n) =& \delta_{\frac{1}{N_1}B_1}\ast\delta_{\frac{1}{N_1N_2}B_2}\ast
\delta_{\frac{1}{N_1N_2N_3}B_3}\ast...\\
=& \mu_n\ast\mu_{>n}
\end{aligned}
$$
with support inside $[0,1]$ always admits an exponential orthonormal basis $E(\Lambda) = \{e^{2\pi i \lambda x}:\lambda\in\Lambda\}$ for $L^2(\mu)$, where $\Lambda$ is obtained from suitably modifying $L_n$.  Here, $\mu_n$ is the  convolution of the first $n$ Dirac measures and $\mu_{>n}$ denotes the tail-term. 

\smallskip

We show that the completeness of $E(\Lambda)$ in general depends on the ``equi-positivity" of the sequence of the pull-backed tail of the Cantor-Moran measure $\nu_{>n}(\cdot) = \mu_{>n}((N_1...N_n)^{-1}(\cdot))$. Such equi-positivity can be analyzed by the integral periodic zero set of the  weak limit of $\{\nu_{>n}\}$.   This result offers a new conceptual understanding of the completeness of exponential functions and it improves significantly many partial results studied by recent research, whose focus has been specifically on $\#B_n\le 4$.

\smallskip

Using the Bourgain's example that a sum of sine can be asymptotically small,  we shows that, 
in the extreme case,  there exists some Cantor-Moran measure such that the equi-positive condition fails and the Fourier transform of the associated $\nu_{>n}$  uniformly converges on some unbounded set. 
\end{abstract}
\maketitle

\section{Introduction}
\begin{definition}
A  Borel probability measure $\mu$ on ${\mathbb R}^d$ is called a {\it spectral measure} if we can find a countable set $\Lambda\subset{\mathbb R}^d$ such that the set of exponential functions $E(\Lambda): = \{e^{2\pi i \lambda \cdot x}:\lambda\in\Lambda\}$ forms an orthonormal basis for $L^2(\mu)$.  If such $\Lambda$ exists, then $\Lambda$ is called a {\it spectrum} for $\mu$.
\end{definition} 

The research of spectral measures was originated from Fuglede \cite{F1974}, whose famous conjecture asserted that $\chi_{\Omega}dx$ is a spectral measure if and only if $\Omega$ is a translational tile. Although the conjecture was disproved eventually \cite{KM1, KM2, T2004}, the problem has led to the development of many new research problems and  determining when a  measure is spectral is still an active research area.

\medskip

In recent years, the study of spectral measures is blooming in the fractal community.  Jorgensen and Pedersen \cite{JP1998} discovered that the standard middle-fourth Cantor measure is a spectral measure. It is the first spectral measure that is non-atomic and  singular to the Lebesgue measure ever discovered. In the same paper, they also showed that  the middle-third Cantor measure is not spectral. Following this discovery, there has been more research on self-similar/self-affine spectral measures \cite{D2012, DHL2013, DHL2014}, as well as the convergence properties of the associated Fourier series \cite{S2000, S2006, DHS2014} and the rescaling properties of a given spectrum \cite{DK2018, DH2016, FHW2018}. 

\medskip

The construction of these fractal spectral measures stem from the existence of Hadamard triples.

\begin{definition}
Let $N\geq2$ be an integer and let $B, L\subset\Z$ be finite sets with $\#L=\#B=M\le N$. We say that the system  $(N, B, L)$ forms a Hadamard triple if the matrix
$$H=\frac{1}{\sqrt{M}}\left[e^{-2\pi i\frac{bl}{N}}\right]_{b\in B, l\in L}$$
is unitary, i.e., $H^*H=I$.
\end{definition}

Soon after Jorgensen-Pedersen's discovery of the first spectral measure, Strichartz \cite{S2000} has already formulated the most general  fractal spectral measures one can possibly generate. Given a sequence of Hadamard triples $\{(N_n,B_n,L_n): n= 1,2,...\}$, one can generate a singular measure without atom using $\{(N_n,B_n)\}$ by 
$$
\mu(N_n,B_n) = \delta_{\frac{1}{N_1}B_1}\ast\delta_{\frac{1}{N_1N_2}B_2}\ast
\delta_{\frac{1}{N_1N_2N_3}B_3}\ast...
$$
where $\delta_A = \frac{1}{\#A}\sum_{a\in A}\delta_a$ and $\delta_a$ denotes  the Dirac measure at the point $a$. We call such measures {\it Cantor-Moran measures} as a generalization of the standard Cantor measure studied first by Moran \cite{Moran}. A natural question here is that 

\medskip

{\bf (Qu 1)}: Given a sequence of Hadamard triples $\{(N_n,B_n,L_n): n= 1,2,...\}$, when is $\mu(N_n,B_n)$ spectral?

\medskip

The Hadamard triple assumption tells us immediately that all Dirac measures in the convolution are actually spectral and we can easily find an infinite mutually orthogonal set of exponential functions using $L_n$. However, the completeness of the exponentials in $L^2(\mu(N_n,B_n))$ is a much harder problem. When all $N_n$ are equal and all $B_n$ are the same set $B$, the Cantor-Moran measure is reduced to the case of  self-similar measure generated by the iterated function system $\left\{f_b(x) = \frac1N (x+b)\right\}_{b\in B}$ with equal probability weights. It has been shown completely that all such self-similar measures  are spectral  by $\L$aba and Wang \cite{LW2002}.  Dutkay, Haussermann and Lai \cite{DHL2017} generalized it to  all self-affine measures in ${\mathbb R}^d$. 

\medskip

General non-self-similar spectral Cantor-Moran measure enriches our understanding of spectral measures. In 2014, An and He \cite{AH2014} showed that if $B_n = \{0,1,...,M_n-1\}$ with $N_n = M_nK_n$, then the resulting Cantor-Moran measure is always spectral. Gabardo and Lai \cite{GL2014} showed that these An-He constructed measures  are exactly  all the measures $\mu$ and $\nu$ that satisfy $\mu\ast\nu ={\mathcal L}_{[0,1]}$, the Lebesgue measure supported on $[0,1]$. In particular, it means that all probability measures that can be convoluted to the Lebesgue measure on [0,1] must be spectral. It offers a generalized tiling perspective of the spectral measures.    These Cantor-Moran measures also show that spectral measures can have support of any Hausdorff dimensions \cite{DS2015}. Last but not least, Cantor-Moran measures offer new examples of fractal measures that admits Fourier frame but not a Fourier orthonormal basis \cite{LW2017}, which leads to a new avenue to study a long-standing problem whether a middle-third Cantor measure has a Fourier frame. 

\medskip

Since then, intensive study on the spectral Cantor-Moran measures have been ongoing,  which attempts to answer the question  ({\bf Qu 1}). It is known however that the answer is negative in general (see Section 9). Nonetheless,  it is widely believed that negative examples are very rare. For instance,  Dutkay and Lai \cite{DL2017} showed that if there are only finitely many Hadamard triples of the form $(N,B_n,L)$, then if we randomly take convolution on these Hadamard triples, then  almost all Cantor-Moran measure are spectral. Furthermore, deterministic positive results have been appeared in many papers (e.g. \cite{AHL2015, AHH2018, HH2017,Shi2018,TF2018}). In all these papers, they all assume either there are only finitely many Hadamard triples in the sequence with a strong assumption on $L$, or $\#B_n\le 4$ (\# denotes cardinality). Except a handful of specific examples, all of the $B_n$ they considered are in $\{0,1,...,N_n-1\}$ . 

\medskip

\subsection{Main Result} In this paper, we focus on our Cantor-Moran measure supported inside $[0,1]$ (i.e. $B_n\subset\{0,1,...,N_n-1\}$). We essentially break through all the unnecessary specific assumptions on $L$ or small number of elements in  $B_n$. One of the main results, expressed in terms of $\#B_n$,  is presented as follows:
\begin{theorem}\label{maintheorem0}
Suppose that $\{(N_n,B_n,L_n)\}$ is a sequence of Hadamard triples with $B_n\subset \{0,1,..,N_n-1\}$ for all $n=1,2,...$. Suppose that $$\liminf_{n\rightarrow\infty}\#B_n<\infty.$$
Then the associated Cantor-Moran measure 
$$
\begin{aligned}
\mu = \mu(N_n,B_n) =& \delta_{\frac{1}{N_1}B_1}\ast\delta_{\frac{1}{N_1N_2}B_2}\ast
\delta_{\frac{1}{N_1N_2N_3}B_3}\ast...\\
\end{aligned}
$$
is spectral and it always admits a spectrum $\Lambda\subset{\mathbb Z}$. 
\end{theorem}

\medskip

Theorem \ref{maintheorem0} will follow directly from Theorem \ref{maintheorem1} and Theorem \ref{maintheorem2}. We now outline the strategy of the proof. Let us set up the notations. For a given sequence of positive integers $N_n\ge 2$ and $B_n\subset\{0,1,...,N_n-1\}$, 
\begin{equation}\label{eq_CM}
\begin{aligned}
\mu = \mu(N_n,B_n) =& \delta_{\frac{1}{N_1}B_1}\ast\delta_{\frac{1}{N_1N_2}B_2}\ast
\delta_{\frac{1}{N_1N_2N_3}B_3}\ast...\\
=& \mu_n\ast\mu_{>n}\\
    \end{aligned}
    \end{equation}
where  $\mu_n$ is the convolutional product of the first $n$ discrete measures and $\mu_{>n}$ is the remaining part.  $\mu$ has support in the compact set
$$K_{\mu}=\left\{\sum_{j=1}^{\infty}\frac{b_j}{N_1\cdots N_j} :\  b_j\in B_j \text{ for all } j\right\}.$$
In particular, $K_{\mu}\subset[0,1]$.
For our further analysis, we will need the measure 
$$
\nu_{>n} (E) = \mu_{>n}((N_1....N_n)^{-1} E),
$$
which is the pull-pack measure of $\mu_{>n}$. If $\mu$ is supported on $[0,1]$, then $\mu_{>n}$ is supported on $[0,(N_1...N_n)^{-1}]$ and $\nu_{>n}$ is the pull back measure from $[0,(N_1...N_n)^{-1}]$ to $[0,1]$. It is also worth to note that if $\mu$ is self-similar, then $\nu_{>n} = \mu$ for all $n$. Let also 
$$
\rho = \frac12\left(\delta_0+\delta_1\right)
$$
be the equal-weighted Dirac mass measure at 0 and 1. For the precise definition of weak convergence of measures, see Section 2.  

\medskip

\begin{theorem}\label{maintheorem1}
Suppose that $\{(N_n,B_n,L_n)\}$ is a sequence of Hadamard triples with $B_n\subset \{0,1,..,N_n-1\}$ for all $n=1,2,...$. Suppose that $\{\nu_{>n}\}$ does not converge weakly to $\rho$.  Then the associated Cantor-Moran measure $\mu(N_n,B_n)$ is spectral and it always admits a spectrum $\Lambda\subset{\mathbb Z}$. 
\end{theorem}



\medskip

We now study the case that $\{\nu_{>n}\}$ converges weakly to $\rho$. One can imagine that the support of the  Cantor-Moran measure is supported in a small neighborhood of the points 0 and 1. We have the following theorem.

\begin{theorem}\label{maintheorem2}
Suppose that $\{(N_n,B_n,L_n)\}$ is a sequence of Hadamard triples with $B_n\subset \{0,1,..,N_n-1\}$ for all $n=1,2,...$. Suppose that $\{\nu_{>n}\}$ converges weakly to $\rho$ and $\liminf_{ n\rightarrow\infty}\#B_n<\infty$.  Then the associated Cantor-Moran measure is a spectral measure. 
\end{theorem}

\begin{remark}

\begin{enumerate}
\item[(1)] Theorem \ref{maintheorem0} follows from Theorem \ref{maintheorem1} and Theorem \ref{maintheorem2} since any $\{\nu_{>n}\}$ either converges or does not converge to $\rho$. Therefore, the rest of the paper will be devoted to proving Theorem \ref{maintheorem1} and \ref{maintheorem2} and studying the remaining delicate case that we will describe in subsection 1.2.
   
   \medskip
   
    \item[(2)] Theorem \ref{maintheorem0} settles completely the spectrality question {\bf (Qu 1)} for Cantor-Moran measures with $\#B_n$ uniformly  bounded in $n$, in particular, for the generalized Bernoulli convolution (all $\#B_n =2$) studied in \cite{HH2017}.
    
    \medskip
    
    \item[(3)] We now sketch the idea of the proof of the theorems. Our main idea is to introduce two important concepts to analyze our probability measures $ \nu_{>n}$. They are called {\it equi-positivity} and {\it admissibility} (see Section 3 for the precise definitions). The following implication will be proved in Section 3. 
\begin{equation}\label{implication}
\{\nu_{>{n_j}}\} \ \mbox{is admissible} \Longrightarrow \{\nu_{>{n_j}}\} \ \mbox{is equi-positive} \Longrightarrow \mu(N_n,B_n)  \ \mbox{is spectral}
\end{equation}
where $\left\{n_j\right\}$ is some subsequence. Theorem \ref{maintheorem1} will correspond exactly to the admissible case. In all other cases, we show that spectrality  holds by  proving the equi-positivity assumption is satisfied.

\medskip

\item[(4)]  Our results also works if  the sequence of Hadamard triples $\{(N_n,B_n, L_n): n=1, 2, \dotsm\}$
is relaxed to an (almost-Parseval) frame triple tower condition as in \cite{LW2017}. After this relaxation, {\it Theorem \ref{maintheorem1}, Theorem \ref{maintheorem2} and Theorem \ref{maintheorem3} below will remain valid with conclusion that the associated Cantor-Moran measure will admit a Fourier frame, instead of an orthonormal basis} (see Section 2).  This reduces the problem of  generating a Fouirer frame for Cantor-Moran measures to construct a frame triple tower.  \end{enumerate}
\end{remark}

\medskip

\subsection{The remaining open case.} We are now left with the case that $\{\nu_{>n}\}$ converges weakly to $\rho$ and $\lim_{ n\rightarrow\infty}\#B_n=\infty$. This is the most delicate case that we cannot resolve  completely. It is directly related to a variant of  a sum of sine problem. To best of our knowledge, it was first studied by Bourgain \cite{B1984}. 

\medskip
\begin{definition}
{\bf (Bourgain's Sum of sine problem)} Let ${\mathcal B}$ be a collection of finite sets of positive integers. We say that the {\it  sum of sine problem holds for ${\mathcal B}$ } if  there exists $\epsilon_0>0$ and $\delta_0>0$ such that for any finite set of positive integers $B\in{\mathcal B}$, 
$$
\max_{x\in\left[0,\frac12-\delta_0\right]}\left|\sum_{b\in B} \sin (2\pi b x)\right| \ge \epsilon_0 (\#B).
$$
\end{definition}

\medskip

To describe our result, we need some extra notations.  For a given  sequence $\{(N_n,B_n)\}$ that generates a Cantor-Moran measure $\mu(N_n,B_n)$ and for any $1/3>\delta>0$, we define
 $$B_{n,\delta} = \{b\in B_n: b/N_n\not\in(\delta,1-\delta)\}$$
 and let 
 $$
 B_{n,\delta,0} = B_{n,\delta}\cap [0,N_n\delta], \   \ B_{n,\delta,1} = N_n- (B_{n,\delta}\cap N_n[1-\delta,1)).
 $$
 We say that $(N_n,B_n)$ is {\it symmetric} if for all sufficiently small $\delta>0$, we can find some $n_0$ such that for all $n\ge n_0$, $B_n = B_{n,\delta}$ and  $B_{n,\delta,1} = B_{n,\delta,0}\setminus\{0\}$. For clarity of the main results in the introduction, we describe only the symmetric case.

\medskip

\begin{theorem}\label{maintheorem3}
Let $\{(N_n,B_n,L_n)\}$ be a sequence of Hadamard triples with $B_n\subset \{0,1,..,N_n-1\}$ for all $n=1,2,...$ with $\{\nu_{>n}\}$ converging weakly to $\rho$ and $\lim_{ n\rightarrow\infty}\#B_n=\infty$. Suppose that $(N_n,B_n)$ is symmetric and there exists a subsequence $\{n_j\}$ such that the sum of sine problem holds for ${\mathcal B} = \{ B_{n_j,\delta,0}\}$, then the associated Cantor-Moran measure is a spectral measure. 
\end{theorem}

One of the main results of Bourgain \cite{B1984} is that  the sum of sine problem cannot hold for all finite sets of positive integers. He showed that there exists finite set of integers $B_n$ of cardinality $n$ such that 
$$
\max_{x\in[0,1]} \left|\sum_{b\in B_n}\sin (2\pi  b x)\right|\le C n^{2/3}
$$
where $C$ is an absolute constant independent of $n$ (see also Kahane's book \cite[p.79]{K1985} for a proof written in English). This result was later generalized to higher dimension \cite{BC2003}. There is also a related {\it cosine minimum conjecture} proposed by Chowla \cite{C1965} in the 1960s, studied by many authors including Bourgain \cite{B1986} and Kolountzakis  \cite{K1994-1,K1994-2},  remaining open as of today. For the history of these sum of sine/cosine problems  and their deep connections to different problems in classical harmonic analysis, one can refer to \cite{K1994-3}. 

\medskip

 As a consequence of the Bourgain's example, it leads us to the following surprising results.

\begin{theorem}\label{maintheorem4}
\begin{enumerate}
    \item There exists Cantor-Moran measure $\mu(N_n,B_n)$ such that the associated measure $\{\nu_{>n}\}$ does not have any equi-positive subsequence. 
    \item There exists Cantor-Moran measure $\mu(N_n,B_n)$ such that the associated measure $\{\nu_{>n}\}$, supported inside $[0,1]$, converges weakly to $\rho$ and $\{\widehat{\nu_{>n}}\}$ converges uniformly to $\widehat{\rho}$ on the non-compact set $\frac12+\Z$.
\end{enumerate}
\end{theorem}

\medskip

Part (i) in Theorem \ref{maintheorem4} shows that our method on checking the equi-positivity cannot be used to determine any (frame-)spectrality for the type of Cantor-Moran measure stated in (i). Part (ii) is perhaps another surprising result from the viewpoint of probability measure theory. It is well-known that the weak convergence of probability measure is equivalent to the uniform convergence of its Fourier transform on all compact sets. This example says that it is possible for a sequence of measures supported inside $[0,1]$ whose Fourier transform  converges on some non-compact sets uniformly. 

\medskip

Bourgain's example was a probabilistic construction. As far as we know, there is no known deterministic construction of the Bourgain's example available. In fact, we also check many  classes of finite set of integers, the sum of sine problems all holds (See Section 8). There may still be hope that the finite set of integers that can generate Hadamard triple satisfy the sum of sine problem, but we do not pursue here in this paper (See Section 9 for details). 




\medskip

\subsection{Organization of the paper.} We now outline the organization of the paper. In Section 2, we will review the notion of weak convergence of probability measures and introduce the known results we will need to prove our theorems. These known basic results can be found in \cite{DHL2017,LW2017}. We notice that all $N_n\ge 2$ and all $B_n\subset\{0,1,...,N_n-1\}$, so that the Cantor-Moran measures are all supported inside $[0,1]$.  

\medskip

In Section 3, we will introduce the two main definitions  {\it equi-positivity} and {\it admissibility} and prove the implication (\ref{implication}).

\medskip

In Section 4, we will show that $\{\nu_{>n}\}$ is admissible if and only if $\{\nu_{>n}\}$ does not converge weakly to $\rho$. Hence, Theorem \ref{maintheorem1}  will be proved.

\medskip

In Section 5, we will study equivalent conditions for non-admissible but equi-positive subsequence. In particular, if $\liminf_{n\rightarrow\infty}\#B_n<\infty$, then $\{\nu_{>n}\}$ still has an equi-positive subsequence. Thus, Theorem \ref{maintheorem2} will be proved as a consequence of Theorem \ref{lemma_bounded_B}.

\medskip

In Section 6, we will study the equivalent conditions for the remaining case that $\{\nu_{>n}\}$ converges weakly to $\rho$ and $\lim_{n\rightarrow\infty}\#B_n=\infty$ to have an equi-positive subsequence. 

\medskip

In Section 7, we will prove Theorem \ref{maintheorem3} and Theorem \ref{maintheorem4} using the Bourgain's example. We will discuss some other related results about the uniform convergence of Fourier transform of probability measures over a non-compact set.

\medskip

In Section 8, we will study the sum of sine problem over a different class of finite sets of integers. We will show that for several fairly large class of finite subsets of integers, the sum of sine problem indeed holds. 

\medskip

In Section 9, we will mention some open problems and  mention how our result can be adapted to more general Cantor-Moran measures whose support is outside $[0,1]$.

\medskip

\section{Notations and known results}
In this section, we will set up our main notation for the rest of our paper. We will also collect some known results that serve as the basis for our proofs. 

\subsection{Measure-theoretic Preliminaries}
Throughout the paper, the Fourier transform of a Borel probability measure $\mu$ on ${\mathbb R}^d$ is defined to be
$$
\widehat{\mu}(\xi) = \int e^{-2\pi i \xi\cdot x} d\mu(x).
$$
Let $K$ be a compact set on ${\mathbb R}^d$. We will consider the following space of functions and measures:

\medskip

 $C_b({\mathbb R}^d)$:  the set of all bounded continuous functions on ${\mathbb R}^d$

\medskip

 ${\mathcal M}(K)$:  the set of all complex Borel  measures supported on compact set $K\subset {\mathbb R}^d$, 
 
 \medskip
 
 ${\mathcal P}(K)$:  the set of all  Borel probability  measures supported on compact set $K\subset {\mathbb R}^d$,

\medskip

 $\widehat{{\mathcal P}}(K) = \{\widehat{\mu}: \mu\in {\mathcal P}(K)\}$.

\medskip

 \noindent It is well-known that $\widehat{{\mathcal P}}(K)\subset C_b({\mathbb R}^d)$. A sequence of measures $\mu_n\in {\mathcal P}(K)$ is said to {\it converge weakly} to a  measure $\mu\in{\mathcal P}(K)$ if for all $f\in C_b({\mathbb R}^d)$, we have 
 $$
 \int f d\mu_n {\longrightarrow} \int f d\mu
 $$
as $n\to\infty$.  For a complete and detailed exposition about weak convergence of probability measures, one may read \cite{C2010}. We now collect all the equivalent conditions about weak convergence in the following lemma. These conditions should be well-known.

\begin{lemma}[\cite{C2010}]\label{lem_weak_equivalence}
The following are equivalent.
\begin{enumerate}
 \item  $\{\mu_n\}$ converges weakly to a probability measure $\mu$. 
    \item For any open set $O$, $\mu(O)\le \liminf_{n\rightarrow \infty} \mu_n(O)$.
    \item For any compact set $K$, $\mu(K)\ge \limsup_{n\rightarrow \infty} \mu_n(K)$.
    \item For any $A$ such that $\mu (\partial A) = 0$, $\lim_{n\rightarrow\infty}\mu_n(A) =\mu(A)$.
    \item The Fourier transform $\{\widehat{\mu}_n(\xi)\}$ converges to $\widehat{\mu}(\xi)$ uniformly on all compact subsets of ${\mathbb R}^d$.
\end{enumerate}
\end{lemma}

Furthermore, it is also known that weak compactness theorem holds: any sequence of probability measures $\{\mu_n\}\subset {\mathcal P}(K)$ has a weakly convergent subsequence $\{\mu_{n_k}\}$ converging to a {\it probability measure} $\mu$. This fact and the lemma will be used  frequently in our exposition. 

\medskip

\subsection{Equicontinuity.} 

A family of functions $\Phi\subset C_b({\mathbb R}^d)$ is called {\it equicontinuous} if for any $\epsilon>0$, there exists $\delta>0$ such that whenever  $\|x-y\|<\delta$, we have $|f(x)-f(y)|<\epsilon$ for all $f\in \Phi$, where $\|\cdot\|$ denotes the Euclidean norm on ${\mathbb R}^d$. The following lemma should be well-known also.

\begin{lemma}\label{equicontinuity}
Let $K$ be a compact set on ${\mathbb R}^d$. Then  $\widehat{{\mathcal P}}(K)$ is equicontinuous.
\end{lemma}

\begin{proof}
Using an elementary inequality $|e^{i\theta}-1| \le |\theta|$ for any $\theta\in{\mathbb R}$, we have that  for any $\mu\in {\mathcal P}(K)$,

\begin{eqnarray*}
|\widehat{\mu}(\xi_1)- \widehat{\mu}(\xi_2)|& =& \left| \int e^{2\pi i \xi_1 \cdot x}-e^{2\pi i \xi_2 \cdot x}d\mu(x)\right| \\
&\le& \int  \left| e^{2\pi i (\xi_1-\xi_2)\cdot x}-1\right|d\mu(x)\\
&\le & 2\pi \|\xi_1-\xi_2\| \int_K \|x\|d\mu(x) \\
&\le &  2\pi C \|\xi_1-\xi_2\|  \ \  \mbox{(since $\mu$ is supported on $K$ and $K$ is bounded)}.\\
\end{eqnarray*}
The equicontinuity follows from this inequality as the upper Lipschitz bound is independent of  $\mu$.
\end{proof}

\medskip

\subsection{(Frame-)spectral Cantor-Moran measure.} 
This subsection follows closely with the framework in \cite{LW2017} and the more  general higher dimensional results presented in \cite{DEL2018}. We say that $\{e^{2\pi i \lambda\cdot x}:\lambda\in \Lambda\}$ forms a {\it Fourier frame} for a Borel probability measure $\mu$ if there exist $0<A\le B<\infty$ such that for all $f\in L^2(\mu)$, 
$$
A \int |f|^2d\mu \le \sum_{\lambda\in\Lambda} |\int f(x)e^{-2\pi i \lambda\cdot x}d\mu(x)|^2\le B\int|f|^2d\mu.
$$
If such Fourier frame exists, $\mu$ is called a {\it frame-spectral measure} and $\Lambda$ is called a {\it frame spectrum} for $\mu$. $A,B$ are respectively called the {\it lower and upper frame bound.} 

\medskip

\begin{definition}
{\rm(i)} Let $N\ge 2$ be a positive integer and let $B, L$ be two finite sets of integers. We say that $(N,B,L)$ forms a {\it frame triple} if there exist  constants $0<c_1\le c_2<\infty$ such that 
$$
c_1\|{\bf w}\| \le \|H{\bf w}\|\le c_2\|{\bf w}\|, \quad \forall\ {\bf w}\in{\mathbb C}^{M}
$$
where $M=\# B$ and matrix
$$H=\frac{1}{\sqrt{M}}\left[e^{-2\pi i\frac{bl}{N}}\right]_{ b\in B, l\in L}.$$ {\rm(ii)} We say that $\{(N_n,B_n,L_n)\}$ is a {\it frame triple tower} if each triple $(N_n, B_n, L_n)$ forms a frame triple with the associated constants $c_1,c_2$ equal to $1-\epsilon_n$ and $1+\epsilon_n$, where $0\le \epsilon_n<1$ and $\sum_{n=1}^{\infty}\epsilon_n<\infty.$ 
\end{definition}

Frame triple is generalized from Hadamard triple since $(N,B,L)$ forms a Hadamard triple if and only if $\|H{\bf w}\| = \|{\bf w}\|$.  When all $\epsilon_n=0$, the frame tower is reduced to the  {\it Hadamdard triple tower} (it is also called the {\it compatible tower} in \cite{S2000}).

\medskip

We will now outline the main theorem that allows us to construct a Fourier basis for the Cantor-Moran measure. For the interest of studying the more general problems about Fourier frame construction in the future, we will state our theorems in terms of frame triple.  



\medskip
We note that by a simple translation, there is no loss of generality to assume $0\in B_n\cap L_n$ for all $n$.  Hence, throughout the rest of the paper until Section 8,  the following will be assumed. 

\medskip

\noindent{\bf Assumption:} For each $n=1, 2, \cdots$, 
\begin{enumerate}
    \item we will assume that $B_n \subset \{0, 1, \cdots, N_n-1\}$ and $0\in B_n$.
    \item there exists $L_n$   with $0\in L_n$ and elements in $L_n$ are in distinct modulo class $(\mod N_n)$  such that $(N_n, B_n, L_n)$ forms a frame triple with bounds $1\pm\epsilon_n$.
\end{enumerate}
   The following theorem is known (see \cite{DHL2017,LW2017}). It is the fundamental theorem on which our analysis is based.

\medskip

\begin{theorem}[\cite{DHL2017, LW2017}]\label{theorem_LW}
Let $\{(N_n,B_n,L_n)\}$ be a  frame triple tower  with  bounds $1\pm\epsilon_n$ and $\sum_{n=1}^{\infty}\epsilon_n<\infty$. Let 
$$
\Lambda_n = L_1+N_1L_2+\cdots+N_1...N_{n-1}L_n, \ \mbox{and} \ \Lambda = \bigcup_{n=1}^{\infty}\Lambda_n.
$$
Suppose that 
$$
\delta(\Lambda): = \inf_{n\ge 1}\inf_{\lambda_n\in\Lambda_n} |\widehat{\mu}_{>n}(\lambda_n)|^2 >0.
$$
Then the Cantor-Moran measure $\mu$ is a frame-spectral measure with a frame-spectrum $\Lambda$ and frame bounds are $\prod_{n=1}^{\infty}(1-\epsilon_n), \prod_{n=1}^{\infty}(1+\epsilon_n)$. In particular, if $\{(N_n,B_n,L_n)\}$ forms a  Hadamard triple tower, then $\mu$ is a spectral measure with a spectrum $\Lambda$.
\end{theorem}

\begin{remark}

\begin{enumerate}
\item Since $\sum_{n=1}^{\infty}\epsilon_n<\infty$, the products $\prod_{n=1}^{\infty}(1\pm\epsilon_n)$ are all finite and thus we have a finite frame bound as in the theorem.
    \item $\delta(\Lambda)>0$ is equivalent to  the condition proposed by Strichartz \cite{S2000}, who originally formulated it as the uniformly separated condition of the points in $\Lambda_n$ from some compact sets.
    \item If $B_n$ is not a subset of $\{0,1,...,N_n-1\}$, we will need an extra measure-theoretic  no-overlap condition (See Section 9). The no-overlap condition is known to be satisfied for $B_n\subset\{0,1,...,N_{n}-1\}$.
\end{enumerate}
\end{remark}

\medskip

\subsection{Factorization of Cantor-Moran measures} For the Cantor-Moran measure given in (\ref{eq_CM}), we can factorize some of the consecutive  factors, so that we have another representation of the same measure. By doing so, we have created a large flexibility for the construction of a (frame-)spectrum using Theorem \ref{theorem_LW}.

\medskip

To perform it precisely, suppose that $\{(N_n,B_n,L_n)\}$ is a sequence of Hadamard triples.  We define ${\bf B}_{n,m}$ to be a set of integers satisfying
$$
\frac{\textbf{B}_{n, m}}{N_n...N_m}=\frac{B_n}{N_n}+\frac{B_{n+1}}{N_nN_{n+1}}+\cdots+\frac{B_{m}}{N_n...N_m}$$
and 
$$
\textbf{L}_{n, m}=L_n+N_nL_{n+1}+\cdots+(N_n...N_{m-1})L_{m}.
$$
 The following lemma is also known, whose proof can be found in \cite[Proposition 3.1]{LW2017}.
 
 \begin{lemma}
 For any $n<m$ and for any $\widetilde{{\bf L}_{n,m}} \equiv {\bf L}_{n,m}$ mod $(N_n\cdots N_m$), then $\{(N_n\cdots N_m, {\bf B}_{n,m}, \widetilde{{\bf L}_{n,m}})\}$ also forms a frame triple tower. 
 \end{lemma}
 
Given a subsequence of positive integers $\{n_k\}$, we define ${\bf N}_{n_1} = N_1\cdots N_{n_1}$, ${\bf N}_{n_k} = N_{n_{k-1}+1}\cdots N_{n_{k}}$, for $k = 2,3,\cdots$. Then the Cantor-Moran measure $\mu = \mu(N_n,B_n)$ can be factorized along this subsequence as 
\begin{equation}\label{eq_factorize}
    \mu = \delta_{\frac{1}{{\bf N}_{n_1}}{\bf B}_{1,n_1}}\ast\delta_{\frac{1}{{\bf N}_{n_1}{\bf N}_{n_2}}{\bf B}_{n_1+1,n_2}}\ast\delta_{\frac{1}{{\bf N}_{n_1}{\bf N}_{n_2}{\bf N}_{n_3}}{\bf B}_{n_2+1,n_3}}\ast\cdots
\end{equation}
Define also 
$$
\mu_{>{\bf N}_{n_k}} = \delta_{\frac{1}{{\bf N}_{n_1}{\bf N}_{n_2}...{\bf N}_{n_{k+1}}}{\bf B}_{n_k+1,n_{k+1}}}\ast\cdots
$$
to be the tail term of $\mu$ by removing the first $k$ factors. Theorem \ref{theorem_LW} can be read as follows:

 \begin{theorem}\label{theorem_LW_factorized}
Let $\{(N_n,B_n,L_n)\}$ be a frame triple tower (or respectively a Hadamard triple tower). Let 
$$
\Lambda_k = {\bf L}_{1,n_1}+{\bf N}_{n_1}{\bf L}_{n_1+1,n_2}+....+{\bf N}_{n_1}...{\bf N}_{n_{k-1}}{\bf L}_{n_{k-1}+1,n_k}, \ \mbox{and} \ \Lambda = \bigcup_{k=1}^{\infty}\Lambda_k.
$$
Suppose that 
$$
\delta(\Lambda): = \inf_{n\ge 1}\inf_{\lambda_k\in\Lambda_k} |\widehat{\mu_{>{\bf N}_{n_k}}}(\lambda_k)|^2 >0.
$$
Then the Cantor-Moran measure $\mu$ is a frame-spectral (or spectral) measure with a frame-spectrum 
(or spectrum) $\Lambda$.
\end{theorem}

 \medskip

\section{Equi-positivity and admissibility}
 Previous section asserts us that a spectrum can be constructed if we can establish $\delta(\Lambda)>0$. In most previous papers, for $\#B_n$ small, authors  check directly the canonical mutually orthogonal sets satisfying  $\delta(\Lambda)>0$. For other cases, $\delta(\Lambda)>0$ is constructed by some random constructions with the assumption on some strong separation conditions (see e.g. \cite{AHL2015}).  The rest of the paper will be devoted to understanding the condition $\delta(\Lambda)>0$ by introducing two conditions that can guarantee $\delta(\Lambda)>0$ can be constructed. These conditions eventually can be studied through classical harmonic analysis theory on the circle group ${\mathbb T}$.

 \subsection{Equi-positivity} The first condition is formulated as below. It was first used for self-affine measure in \cite{DHL2017}.
 
\begin{definition}
Let $\Phi$ be a collection of probability measures on   compact set $[0, 1]$. We say that $\Phi$ is {\it an equi-positive family} if there exists $\epsilon_0>0$ such that for all   $ x\in [0,1]$ and for all $\nu\in \Phi$, there exists $k_{x,\nu}\in{\mathbb Z}$ such that
$$
|\widehat{\nu}(x+k_{x,\nu})|\ge \epsilon_0.
$$
\end{definition}

\medskip

Equi-positivity is also equivalent to saying that we can find a fundamental domain $K_{\nu}$ of ${\mathbb Z}$ so that $|\widehat{\nu}|$ is always away from 0 at least  an $\epsilon_0>0$ on $K_{\nu}$, where $\epsilon_0$ is independent of $\nu$. 
\medskip

\begin{theorem}\label{theorem_admissible_spectral}
Let $\mu(N_n,B_n)$ be a Cantor-Moran measure with  $\{(N_n,B_n,L_n)\}$ forming a frame triple tower (or respectively a Hadamard triple tower). Suppose that there exists a subsequence $\{n_k\}$ such that $\{\nu_{>n_{k}}\}$ is equi-positive. 
Then $\mu(N_n,B_n)$ is a frame-spectral (or spectral) measure with a frame-spectrum (or spectrum) in ${\mathbb Z}$.
\end{theorem}

\begin{proof}
We first factorize $\mu(N_n,B_n)$ along with the subsequence  given in the assumption, so that it has the form 
\begin{equation}\label{eq_factorized2}
    \mu = \delta_{\frac{1}{{\bf N}_{n_1}}{\bf B}_{1,n_1}}\ast\delta_{\frac{1}{{\bf N}_{n_1}{\bf N}_{n_2}}{\bf B}_{n_1+1,n_2}}\ast\delta_{\frac{1}{{\bf N}_{n_1}{\bf N}_{n_2}{\bf N}_{n_3}}{\bf B}_{n_2+1,n_3}}\ast\cdots
\end{equation}
where ${\bf N}_{n_1} = N_1...N_{n_1-1}$, ${\bf N}_{n_k} = N_{n_{k-1}}...N_{n_{k}-1}$, for $k = 2,3,\cdots$. Moreover, we also know that each $({\bf N}_{n_k}, {\bf B}_{n_{k-1}, n_{k}-1}, {\bf L}_{n_{k-1}, n_{k}-1)}$ forms a frame triple. We can further factorize consecutive factors in (\ref{eq_factorized2})  of $\mu$ if necessary, so that we can choose that ${\bf N}_{n_k}$ as large as we want. With respect to the factorization, we let 
$$
\mu_{>{\bf N}_{n_k}} = \delta_{\frac{1}{{\bf N}_{n_1}{\bf N}_{n_2}...{\bf N}_{n_{k+1}}}{\bf B}_{n_k+1,n_{k+1}}}\ast\cdots
$$
be the tail term of $\mu$ by removing the first $k$ factors and let 
$$
\nu_{>{{\bf N}_{n_k}}} (E) = \mu_{>{\bf N}_{n_k}} \left(\frac{E}{{\bf N}_{n_1}...{\bf N}_{n_k}}\right), \ \forall  \ E  \ \mbox{Borel}
$$
be the pull back measure of $\mu_{>{\bf N}_{n_k}}$ onto $[0,1]$. We can now undo the factorization and we notice that the measure 
\begin{equation}\label{equndo}
\nu_{>{{\bf N}_{n_k}}} = \nu_{>n_k}.
\end{equation}
Our goal is to construct inductively a sequence of sets $\Lambda_k$ from ${\bf L}_{n_{k-1}, n_{k}-1}$ such that 
$$
\delta(\Lambda) = \inf_{k\ge 1}\inf_{\lambda\in \Lambda_k} |\widehat{\mu_{>{\bf N}_{n_k}}}(\lambda_{k})|^2 >0.
$$
\medskip

 It follows from Lemma \ref{equicontinuity} that $\{\widehat{\nu_{>n}}\}$ is an equicontinuous family. The equi-positivity of $\{\nu_{>n_k}\}$ and the equicontinuity of $\{\widehat{\nu_{>n_k}}\}$ imply that there exists $\epsilon_0>0$ and $\delta_0>0$ such that for all $x\in [0,1]$ and for all $\nu_{>n_k}$, there exists $k_{x,\nu_{>{n_k}}}\in{\mathbb Z}$ such that
\begin{equation}\label{eq_equi-positive-continuous}
|\widehat{\nu_{>n_k}}(y+x+k_{x,\nu_{>n_k}})|\ge \epsilon_0    
\end{equation}
whenever $|y|<\delta_0$. Also, if $x =0$, we can take $k_{x,\nu_{>n_k}} = 0$ since $\widehat{\nu_{>n_k}}( 0)=1$. 

\medskip

We now  construct inductively $\Lambda_k$ so that $\delta(\Lambda)>0$. First, we take $\Lambda_0  = \{0\}$. Suppose that $\Lambda_{k-1}$ has been constructed and it satisfies
$$
\inf_{\lambda_{k-1}\in\Lambda_{k-1}} |\widehat{\mu_{>{\bf N}_{n_{k-1}}}}(\lambda_{k-1})|^2\ge \epsilon_0^2>0,
$$
where $\epsilon_0$ is given in (\ref{eq_equi-positive-continuous}). We can  take a large enough $n_{k}$ in the subsequence and factorize more levels of the Dirac measures so that we obtain a large enough ${ \mathbf N}_{n_k}$ with the following happen: 
\begin{equation}\label{eq_small_delta}
\sup_{\lambda_{k-1}\in\Lambda_{k-1}}\left|\frac{1}{{\bf N}_{n_1}{\bf N}_{n_2}...{\bf N}_{n_k}}\lambda_{k-1}\right| < \delta_0.
\end{equation}
We now define 
$$
\Lambda_k = \Lambda_{k-1}+\{ {\bf N}_{n_1}{\bf N}_{n_2}...{\bf N}_{n_{k-1}} \ell_k+ {\bf N}_{n_1}{\bf N}_{n_2}...{\bf N}_{n_k} k_{x_{\ell_k},\nu_{>n_{k}}}: \ell_k\in{\bf L}_{n_{k-1+1}, n_k}\},
$$
where $x_{\ell_k} = \ell_{k}/{\bf N}_{n_k}\in[0,1]$ so that $k_{x_{\ell_k},\nu_{>n_k}}$ is defined as in  (\ref{eq_equi-positive-continuous}). Now, writing 
$$
\lambda_k = \lambda_{k-1}+{\bf N}_{n_1}{\bf N}_{n_2}...{\bf N}_{n_{k-1}} \ell_k+ {\bf N}_{n_1}{\bf N}_{n_2}...{\bf N}_{n_k} k_{x_{\ell_k},\nu_{>n_{k}}}, 
$$
for some $\lambda_{k-1}\in \Lambda_{k-1}$, we have  
$$
\begin{aligned}
|\widehat{\mu_{>{\bf N}_{n_{k}}}}(\lambda_{k})|^2 =& \left|\widehat{\nu_{>{\bf N}_{n_{k}}}}\left(\frac{\lambda_{k}}{{\bf N}_{n_1}{\bf N}_{n_2}...{\bf N}_{n_k}}\right)\right|^2\\
=&\left|\widehat{\nu_{>n_{k}}}\left(\frac{1}{{\bf N}_{n_1}{\bf N}_{n_2}...{\bf N}_{n_k}}\lambda_{k-1}+x_{\ell_k}+ k_{x_{\ell_k},\nu_{>n_k}}\right)\right|^2 \ (\mbox{using} \ (\ref{equndo}))\\
\ge& \epsilon_0^2>0
\end{aligned}
$$
by  (\ref{eq_equi-positive-continuous}) and \eqref{eq_small_delta}. Hence, we have $\delta(\Lambda)\ge \epsilon_0^2>0$ for $\Lambda = \bigcup_{k=1}^{\infty}\Lambda_k$ and our proof is complete.  
\end{proof}

\medskip

\subsection{Admissible Family.} Theorem \ref{theorem_admissible_spectral} tells us that the existence of an equi-positive subsequence $\{\nu_{>n_k}\}$ is enough to show that the Cantor-Moran measure $\mu(N_n, B_n)$ we consider is spectral. Next, we will focus on proving such equi-positive sequence exists by introducing admissible family. We first need the following definition.

\medskip

\begin{definition}\label{definition_periodic-set}
For any Borel probability measure $\mu$ on ${\mathbb R}^d$, the {\it integral periodic zero set} is defined to be the set
$$
{\mathcal Z}(\mu) = \{\xi\in{\mathbb R}: \widehat{\mu}(\xi+k) = 0, \ \forall\ k\in{\mathbb Z}\}.
$$
We say that a family of measures $\Phi\subset{\mathcal P}([0,1])$ is an {\it admissible family} if for all $\nu\in\Phi$, ${\mathcal Z}(\nu)= \emptyset$ and for any possible weak limits of $\Phi$, their integral periodic zero sets are also emptysets. We say that a sequence of measures $\{\nu_n\}$ is an {\it admissible sequence} if $\{\nu_n\}$ converges weakly  and it forms an admissible family. 
\end{definition}
\medskip

\begin{proposition}\label{prop1}
Let $\Phi$ be an admissible family of measures in ${\mathcal P}([0, 1])$ and let $x\in{\mathbb R}$. Then there exists $\epsilon_x>0$ such that for all $\nu\in\Phi$
$$
\sup \{|\widehat{\nu} (x+k)|: k\in{\mathbb Z}\} >\epsilon_x.
$$
\end{proposition}

\begin{proof}
Suppose that the conclusion is false. Then for any $\epsilon>0$, there exists $\nu_{\epsilon}\in\Phi$ such that
$$
\sup \{|\widehat{\nu_{\epsilon}} (x+k)|: k\in{\mathbb Z}\} \le \epsilon.
$$
This means that $|\widehat{\nu_{\epsilon}}(x+k)|\le\epsilon$ for all $k\in{\mathbb Z}$. Note that by passing subsequence if necessary, $\{\nu_{\epsilon}\}$ converges weakly to some probability measure $\nu_{0}$. By the admissibility assumption of $\Phi$, ${\mathcal Z}(\nu_{0}) = \emptyset$. However, we have $|\widehat{\nu_{\epsilon}} (x+k)|\le\epsilon$ for all $k\in{\mathbb Z}$. 
As $\{\widehat{\nu_{\epsilon}}\}$ converges pointwisely,  we  have that $\widehat{\nu_0}(x+k) = 0$ for all $k\in{\mathbb Z}$ and thus ${\mathcal Z}(\nu_0)$ is non-empty. This is a contradiction. Therefore, our conclusion holds.
\end{proof}

The following is the key theorem that we will use to construct our spectrum for an admissible family.

\begin{theorem} \label{Key_proposition}
Let $\Phi$ be an admissible family of measures in ${\mathcal P}([0, 1])$. Then $\Phi$ is equi-positive.
\end{theorem}

\begin{proof}
 We need to show that there exists $\epsilon_0>0$  such that for all $x\in [0,1]$ and for all $\nu\in\Phi$, there exists $k_{x,\nu}\in{\mathbb Z}$ such that
$$
|\widehat{\nu}(x+k_{x,\nu})|\ge \epsilon_0.
$$

For any $x\in [0,1]$, we take $\epsilon_x$ as in Proposition \ref{prop1}. Then for any $\nu\in \Phi$, we can find $k_{x,\nu}$ such that
$$
|\widehat{\nu}(x+k_{x,\nu})|\ge \epsilon_x.
$$
By Lemma \ref{equicontinuity}, $\Phi$ is  equicontinous on ${\mathbb R}$. we can find  $\delta_x$ such that for all $|y|\le \delta_x$, we have
$$
|\widehat{\nu}(x+y+k_{x,\nu})|\ge \frac{\epsilon_x}{2},  \quad \forall\ \nu\in \Phi.
$$
As $[0, 1] \subset \bigcup_{x\in X} B(x,\delta_x/2)$, by the compactness of $[0, 1]$, we can find $x_1,...,x_N\in [0, 1]$ such that $[0, 1] \subset B(x_1,\delta_{x_1}/2)\cup...\cup B(x_N,\delta_{x_N}/2)$. We now take
$$
\delta_0 = \min\left\{\frac{\delta_{x_j}}{2}: j=1,...,N\right\},  \
\epsilon_0 = \min\left\{\frac{\epsilon_{x_j}}{2}: j=1,...,N\right\}.
$$
Now, $\delta_0$ and $\epsilon_0$ are positive and independent of $x\in [0, 1]$ and $\nu\in \Phi$. We claim that the stated property holds. Indeed, for any $x\in [0, 1]$, $x\in B(x_j,\delta_{x_j}/2)$ for some $j=1,...,N$. 
Hence,
$$
|\widehat{\nu}(x+k_{x_j,\nu})| = |\widehat{\nu}(x_j+(x-x_j)+k_{x_j,\nu})|\ge \frac{\epsilon_{x_j}}{2}\ge \epsilon_0.
$$
Therefore, we just redefine $k_{x,\nu} = k_{x_j,\nu}$ to obtain our desired conclusion. 
\end{proof}

In particular, we have the following theorem. It follows from Theorem \ref{Key_proposition} and Theorem \ref{theorem_admissible_spectral}. It assumes a stronger condition than equi-positivity, but it will be useful for our later analysis.

\begin{theorem}\label{thadmissible}
Let $\mu(N_n ,B_n)$ be a Cantor-Moran measure with  $\{(N_n,B_n,L_n)\}$ forming a frame triple tower (respectively a Hadamard triple tower). Suppose that there exists a subsequence $\{n_k\}$ such that $\{\nu_{>n_{k}}\}$ is an admissible sequence. Then $\mu(N_n,B_n)$ is a frame-spectral (respectively specrtral) measure.  
\end{theorem}

\medskip

\section{Admissible family of Cantor-Moran measures}

\subsection{General admissible family.} In this section, we will study the admissibiliy condition of  $\mu \in {\mathcal P}[0,1]$. We will identify the circle group ${\mathbb T}$ as $[0,1)$. If a measure $\nu\in{\mathcal M}([0,1])$ has the property that $\nu\{0\}=0$ or $\nu\{1\} = 0$. Then $\nu$ can be regarded as a measure on ${\mathbb T}$ by an obvious identification. 

\medskip

We now give a complete characterization for which kind of measures in ${\mathcal P}[0,1]$ so that ${\mathcal Z}(\mu) = \emptyset$. We also recall a well-known fact in classical harmonic analysis (see e.g. \cite[p.35]{K2004}).

\begin{lemma}\label{lemma_uniqueness}(uniqueness of Fourier coefficients)
Let $\nu\in{\mathcal M}({\mathbb T})$. Suppose that $\nu(k) = 0$ for all $k\in{\mathbb Z}$. Then $\nu = 0$. 
\end{lemma}










\medskip

\begin{proposition}\label{prop_zero}
Let $\mu\in {\mathcal P}([0,1])$ and suppose that $\mu(\{0\})=0$ or $\mu(\{1\}) = 0$. Then ${\mathcal Z}(\mu) = \emptyset$.
\end{proposition}

\begin{proof}
Suppose that there exists $\xi\in{\mathcal Z}(\mu)$.  Define a measure $\nu$ by $d\nu(x)=e^{-2\pi i\xi x}d\mu(x)$, which is nonzero since $\mu$ is a Borel probability measure and $\nu$ is absolutely continuous with respect to $\mu$ with a non-zero density. Now,  $\nu(\{0\})=0$ or  $\nu(\{1\})=0$ by the assumption and thus $\nu$ can be regarded as a measure on ${\mathbb T}$. Moreover, its Fourier coefficient
$$\widehat{\nu}(k)=\widehat{\mu}(\xi+k)=0$$
for all $k\in\Z$. By Lemma \ref{lemma_uniqueness}, $\nu$ is a zero measure, which is a contradiction.
\end{proof}

%


\medskip

\begin{remark}\label{remark}The assumption that $\mu(\{0\}) = 0$ or $\mu(\{1\}) = 0$ cannot be removed from the proposition. For example, we consider $\rho = \frac{1}{2}(\delta_0+\delta_1)$. 
Then
$
\widehat{\rho}(\xi) = \frac{1+e^{-2\pi i \xi }}{2},
$
and $\widehat{\rho}(1/2+k)=0$ for all $k\in{\mathbb Z}$. Therefore, $1/2\in{\mathcal Z}(\rho)$.
\end{remark}

\medskip

The following theorem shows however that $\rho$ is the only possible exception.\begin{theorem}
\label{prop_Dirac_measure}
Let $\mu\in {\mathcal P}([0,1])$. Then ${\mathcal Z}(\mu)\neq\emptyset$ if and only if $\mu = \rho=\frac{1}{2}(\delta_0+\delta_1)$.
\end{theorem}
\begin{proof} From Remark \ref{remark}, we just need to prove the necessity. From Proposition \ref{prop_zero}, if ${\mathcal Z}(\mu)\neq\emptyset$, then $\mu(\{0\})>0$ and $\mu(\{1\})>0$. So we may write
$$\mu=p_0\delta_0+p_1\delta_1+\mu_1$$
where $\mu_1$ is a finite Borel measure on $[0, 1]$ and $\mu_1(\{0, 1\})=0$. Then
$$\widehat{\mu}(\xi)=p_0+p_1e^{-2\pi i \xi}+\widehat{\mu_1}(\xi).$$
Let $\xi_0$ be an element in ${\mathcal Z}(\mu)$, then
$$\widehat{\mu}(\xi_0+k)=p_0+p_1e^{-2\pi i \xi_0}+\widehat{\mu}_1(\xi_0+k)=0, \quad \forall\ k\in \Z.$$
It implies that
$$\widehat{\mu}_1(\xi_0+k)=-p_0-p_1e^{-2\pi i \xi_0} := c.$$
Consider the complex measure on [0, 1]
$$d\nu(x)=e^{-2\pi i \xi_0 x}d(\mu_1-c\delta_0)(x).$$
Since $(\mu_1-c\delta_0)(\{1\})=0$,  we have $\nu(\{1\})=0$. Also,  $\widehat{\nu}(k)=\widehat{\mu_1}(\xi_0+k)-c=0$ for all $k\in\Z$. By regarding   $\nu$ as a measure on ${\mathbb T}$, from Lemma \ref{lemma_uniqueness}, $\nu=0$ which implies that $\mu_1 = c\delta_0$. However, $\mu_1$ has no measure at the point $0$. Thus, $c = 0$ and $\mu_1 = 0$.  This shows that $c = p_0+p_1e^{-2\pi i\xi_0}=0$. The equation $p_0+p_1e^{-2\pi i\xi_0}=0$ is equivalent to $e^{-2\pi i\xi_0}=-\frac{p_0}{p_1}$. The left hand has modulus 1, so $p_0=p_1=\frac12$ and thus $\mu = \rho$ follows. 
\end{proof}

\medskip

\subsection{Admissible family of Cantor-Moran measures}

Let $\{N_n\}$ be a sequence of non-negative integers with $N_n\ge 2$. Let $B_n\subset\{0,1,..,N_{n}-1\}$ such that $\#B_n \le N_n$. We form the associated Cantor-Moran measure by
$$
\begin{aligned}
\mu = \mu(N_n,B_n) =& \delta_{\frac{1}{N_1}B_1}\ast\delta_{\frac{1}{N_1N_2}B_2}\ast\delta_{\frac{1}{N_1N_2N_3}B_2}\ast...\\
=& \mu_n\ast\mu_{>n}.\\
\end{aligned}
$$

\medskip

   As Cantor-Moran measure is purely singular without atoms,  $\mu (\{1\} )= 0$ and  by  Proposition \ref{prop_zero}, ${\mathcal Z}(\mu) = \emptyset$. The following question, if true, would be enough to show that  the above Cantor-Moran measure, if  it can form a frame triple, is a frame-spectral measure.

\medskip

{\bf Question:}  Let $\mu(N_n,B_n)$ be a Cantor-Moran measure,  can we find a subsequence $\{n_k\}$ such that  $\Phi : = \{\nu_{>{n_k}}: k=1,2,...\}$ forms an admissible sequence?

\medskip

Proposition \ref{prop_Dirac_measure} offers a simple solution to the above question.

\begin{lemma}\label{Lemma_rho}
$\{\nu_{>n}\}$ forms an admissible family if and only if $\{\nu_{>n}\}$ does not converge to $\rho$ weakly. 
\end{lemma}

\begin{proof}
It would be easier to prove $\{\nu_{>n}\}$ does not form an admissible family if and only if $\{\nu_{>n}\}$ converges to $\rho$ weakly. 
Suppose that $\{\nu_{>n}\}$ does not form an admissible family. By the weak compactness of ${\mathcal M}[0,1]$. Any subsequence $\{\nu_{>n_k}\}$ has a weakly convergent subsequence. Now, this convergent subsequence must converge weakly to $\rho$ since $\{\nu_{>n}\}$ is not admissible but ${\mathcal Z}(\nu_{>n}) = \emptyset$ for all $n$, Theorem \ref{prop_Dirac_measure} tells us that  the only weak limit can be $\rho$. We have shown that any subsequence has a subsequence converging to $\rho$. It implies that $\{\nu_{>n}\}$ converges weakly to $\rho$.

\medskip

Conversely, if $\{\nu_{>n}\}$ converges to $\rho$ weakly, then  $\rho$ is the only weak limit of $\{\nu_{>n}\}$. Hence,   $\{\nu_{>n}\}$ cannot be admissible.
\end{proof}

\medskip

As we will see in the following two propositions, for a large class of Cantor-Moran measures, we can always find an admissible subsequence. For $\nu\in{\mathcal M}[0,1]$, we let $[0,c_{\nu}]$ be the convex hull  of the support of $\nu$. 

\medskip

\begin{proposition}\label{prop4.3}
 Let $\mu(N_j,B_j)$ be a Cantor-Moran measure. Suppose that  $c: = \sup_n c_{\nu_{>n}}<1$. 
 Then $\Phi : = \{\nu_{>n}: n=1,2,...\}$ forms an admissible sequence.
\end{proposition}

\begin{proof}
As each of the $\nu_{>n}$ is a Cantor-Moran measure generated by $\{(N_{n+j},B_{n+j})\}$, it is a singular measure without any atoms. In particular, ${\mathcal Z}(\nu_{>n}) = \emptyset$. Note that the support of $\nu_{>n}$ is contained in the interval $[0,c]$ by our assumption. Hence, $\nu_{>n}[1-c,1+\epsilon] = 0$ for all $n$ and $\epsilon>0$. If $\nu$ is a weak limit of $\{\nu_{>n}\}$, then Lemma \ref{lem_weak_equivalence} implies that 
$$
\nu(1-c,1+\epsilon) \le \liminf_{n\rightarrow\infty} \nu_{>n}(1-c,1+\epsilon) = 0.
$$
Hence, $\nu(\{1\}) = 0$. It now follows that ${\mathcal Z}(\nu) = \emptyset$ by Proposition \ref{prop_zero}. This shows that the family is admissible.
\end{proof}

\medskip

\begin{proposition}\label{prop4.6}
Suppose that $M: = \sup N_j<\infty$. Then there exists a subsequence $\{n_k\}$ such that $\{\nu_{>n_k}\}$ is an admissible sequence. 
\end{proposition}

\begin{proof}
We first assume that there are only finitely many $n$ such that  $N_{n}-1\in B_n$. We can find an $n_0$ such that $N_n-1\not\in B_n$ for all $n\ge n_0$ .  Then for all $n\ge n_0$, 
$$
\begin{aligned}
c_{\nu_{>n}} \le& \sum_{j=1}^{\infty} \frac{N_{n+j}-2}{N_{n+1}....N_{n+j}} \\
\le & 1-\frac{1}{M}-\frac{1}{M^2}-... = \frac{M-2}{M-1}<1.
\end{aligned}
$$
Hence, by Proposition \ref{prop4.3}, $\{\nu_{>n}: n\ge n_0\}$ is an admissible sequence. 

\medskip

We now suppose that there are infinitely many $n$ such that $N_n-1\in B_n$. Take the subsequence $\{n_k\}$ such that $N_{n_k} = N\le M$. Consider $I = \left[1-\frac1N, 1-\frac1N+\frac{1}{2N}\right]$. Then $I$ contains at least an interval $\left[1-\frac1N, 1-\frac1N+\frac{1}{N N_{n_k+1}}\right]$ and
$$
\nu_{>n_k-1}(I) \ge \frac{1}{N_{n_k}N_{n_k+1}}\ge \frac{1}{M^2}.
$$
Hence, if $\nu$ is a weak limit of $\{\nu_{>n_k-1}\}$,
$$
\nu(I) \ge \limsup_{k\rightarrow \infty} \nu_{>n_k-1}(I)\ge \frac{1}{M^2}.
$$
This shows that the weak limit cannot be the measure $\rho$ since an interval away from 1 has a positive measure. In particular, the integral periodic zero set is empty by Theorem \ref{prop_Dirac_measure}.
\end{proof}

\medskip

These propositions show that as long as the measure stays away from 1 or $N_j$ is not growing up, all the resulting Cantor-Moran measures are spectral. Now, the proof of Theorem \ref{maintheorem1} in apparent.

\bigskip

\noindent{\bf Proof of Theorem \ref{maintheorem1}.} By Lemma \ref{Lemma_rho}, the assumption in Theorem \ref{maintheorem1} implies that $\{\nu_{>n}\}$ is an admissible family. Hence, the spectrality or frame-spectrality follows from Theorem \ref{thadmissible}. \eproof 

\bigskip

However, non-admissible Cantor-Moran measures exist.

\begin{example}
Let $N_n = 2^{2n}$ and let $B_n = \{0, 2^{2n}-1\}$. Then $(N_n,B_n,L_n)$ forms a Hadamard triple with $L_n = \{0, 2^{2n-1}\}$. Moreover, $\{\nu_{>n}\}$ converges weakly to $\rho$.
\end{example}

\begin{proof}
The fact that it is a Hadamard triple follows from a direct check. For the weak convergence, we note that the support of $\nu_{>n}$ is contained in $[0,2^{-{2n}}]\cup [1-2^{-2n},1]$. Hence, for all $\delta>0$, $\lim_{n\rightarrow\infty}\nu_{>n}(\delta,1-\delta) = 0$. Hence, any weak limit of $\{\nu_{>n}\}$ must be supported on $\{0,1\}$. But with $n$ sufficiently large,  $\nu_{>n}[0,\delta] =\frac12$. This shows that the weak limit must be $\rho$.
\end{proof}

\medskip

In fact, as long as we take $N_n>> \#B_n$ with $B_n$ concentrating very closely at $0, N_n-1$, we can easily construct Cantor-Moran measures $\{\nu_{>n}\}$ converging weakly to $\rho$. In this case, admissibility condition fails.   
\medskip

\section{Non-admissible Cantor-Moran measures (I): $\liminf_{n\rightarrow\infty} \# B_n <\infty$.}
\subsection{Non-admissible but equi-positive family.} Theorem \ref{maintheorem1} tells us that we are left with the case where $\{\nu_{>n}\}$ converges weakly to $\rho$ or equivalently, admissible subsequence is not available. In this case, we need to study the validity of the equi-positivity condition.

\medskip

 
  From this section and on, we will focus on the situation that $\{\nu_{>n}\}$ converges weakly to $\rho$.  In order to use Theorem \ref{theorem_admissible_spectral} to show the spectrality of the Cantor-Moran measure   $\mu(N_n,B_n)$, we need to establish the existence of the equi-positive subsequence.  The following proposition captures all the equivalent conditions  we need to study.
  
  \medskip
  
  \begin{proposition}\label{prop_equiv-positive}
  Suppose that $\{\nu_{n}\}$ converges weakly to $\rho$. Then the following are equivalent.
  \begin{enumerate}
      \item   there exists an equi-positive subsequence $\{\nu_{n_j}\}$.
      \item $\{\widehat{\nu_{n}}\}$ does not converge uniformly to $\widehat{\rho}$ on $\frac12+\Z$. 
      \item   there exists a subsequence $\{\nu_{n_j}\}$ such that the following property holds: there exists $\epsilon_0>0$,  for any $j\ge 1$, we can find $k_j$ such that
      $$\left|\widehat{\nu_{n_j}}(\frac12+k_j)\right|\ge \epsilon_0.$$
  \end{enumerate}
  \end{proposition}
  
  \begin{proof}
That  (i) implies (ii) is just the definition of equi-positivity at $x= \frac12$.   As $\{\nu_{n}\}$ weakly converges to $\rho$ and $\widehat{\rho}(\frac12+k) = 0$ for all $k\in\Z$. In view of this, (ii) and (iii) are equivalent by the definition of uniform convergence.  
  
  \medskip
  
  We now suppose  that (iii) (or equivalently (ii)) holds. Then  there exists a subsequence $\{\nu_{n_j}\}$ such that the following property holds: there exists $\epsilon_0>0$,  for any $j\ge 1$, we can find $k_j$ such that
      $$
      \left|\widehat{\nu_{n_j}}(\frac12+k_j)\right|\ge \epsilon_0.
      $$
      By the equicontinunity of  $\widehat{\nu_{n_j}}$, we can find $\delta_0>0$, independent of $j$, such that 
      $$
      \left|\widehat{\nu_{n_j}}(\frac12+x+k_j)\right|\ge \frac{\epsilon_0}{2}.
      $$
      for all $|x|\le \delta_0$. We can take $k_{x,\nu_{n_j}} = k_j$ for $x\in\left[\frac12-\delta_0, \frac12+\delta_0\right]$.  On the other hand, 
      it is known that $\{\widehat{\nu_{n_k}}\}$ converges uniformly to $\widehat{\rho}(\xi)$ on $[0, 1]$. Therefore, for the  positive constant $\frac{1}{2}\sin(\pi\delta_0)$, there exists $J>0$ such that for all $j\geq J$ we have 
\begin{equation*}\label{low-bdd2}
    |\widehat{\nu_{n_j}}(x)-\widehat{\rho}(x)|<\frac{1}{2}\sin(\pi\delta_0), \quad x\in[0, 1].  
\end{equation*}
Note that $\widehat{\rho}(x) = e^{\pi i x} \cos (\pi x)$. Therefore,  for $x\in[0, 1]\setminus\left[\frac12-\delta_0, \frac12+\delta_0\right]$, we have 
$$|\widehat{\nu_{n_j}}(x)|\geq |\cos\pi x|-\frac{1}{2}\sin(\pi\delta_0)\geq\frac{1}{2}\sin(\pi\delta_0).$$
 Hence, we can let $k_{x,\nu_{n_j}} = 0$ for all $x\in[0, 1]\setminus\left[\frac12-\delta_0, \frac12+\delta_0\right]$. This shows that 
 $\{\nu_{n_j}\}_{j=J}^{\infty}$ is an equi-positive subsequence with uniform lower bound  $\min\left\{\epsilon_0/2,\frac{1}{2}\sin(\pi\delta_0)\right\}$ in the definition of equi-positivity.
  \end{proof}

  \medskip
  
   \subsection{The case $\liminf_{n\rightarrow\infty}\#B_{n}<\infty$.} In this case,  we can find a positive integer $M>0$ and a subsequence $\{n_j\}$ such that $\sup_{j} \#B_{n_j} \le M$.   Here we need the following Wiener theorem  concerning how the discrete part of a Borel measure $\mu$ on ${\mathbb T}$ can be ``recovered" from its Fourier-Stieltjes series.
 
 \medskip
 
 \begin{theorem}\label{discrete-inverse-theorem}
Let $\mu$ be any complex Borel measure on ${\mathbb T}$. Then
$$
\sum_{\tau} |\mu(\{\tau\})|^2 = \lim_{N\rightarrow\infty} \frac{1}{2N+1} \sum_{k= -N}^{N}|\widehat{\mu}(k)|^2,
$$
where the sum on the left is taken  over all the atoms of $\mu$. 
\end{theorem}

\medskip
The following lemma will be needed.
\medskip

\begin{lemma}\label{lemma_discrete}
 Suppose that $\{\nu_{>n}\}$ converges weakly to $\rho$. If there exist $\epsilon_0>0$, $\delta_0>0$ and  a subsequence $\{n_j\}$ such that for all $j\geq1$, we can find integers $k_{j}$ such that $\frac{k_{j}}{N_{n_j}}\in[0,1]\setminus\left(\frac12-\delta_0,\frac12+\delta_0\right)$ and 
$$
\left|\widehat{\delta_{B_{n_j}/N_{n_j}}}(\frac12+k_{j})\right|\ge \epsilon_0.
$$
Then $\{\nu_{>n}\}$ has an equi-positive subsequence.
\end{lemma}

\begin{proof}
Suppose that we can find integers $k_{j}$ such that $\frac{k_{j}}{N_{n_j}}\in[0,1]\setminus\left(\frac12-\delta_0,\frac12+\delta_0\right)$ and
 $$
 \left|\widehat{\delta_{B_{n_j}/N_{n_j}}}(\frac12+k_{j})\right|\ge \epsilon_0.
 $$
 As we have
 $$
\widehat{\nu_{>(n_j-1)}}(\frac12+k_{j}) = \widehat{\delta_{B_{{n_j}}/N_{n_j}}} (\frac12+k_{j}) \cdot \widehat{\nu_{>(n_j)}}\left(\frac{\frac12+k_{j}}{N_{n_j}}\right),
$$
with $\frac{\frac12+k_{j}}{N_{n_j}} \in [0,1]\setminus \left(\frac12-\frac{\delta_0}{2},\frac12+\frac{\delta_0}{2}\right)$ for $j$ large enough. Using the fact that $\{\widehat{\nu_{>n_j}}\}$ converges uniformly to $\widehat{\rho}$ on $[0,1]\setminus (\frac12-\delta_0,\frac12+\delta_0)$, it means that we can find $\epsilon_1$ independent of $j$ such that  for $j$ large enough,
$$\left|\widehat{\nu_{>n_j}}(\frac{\frac12+k_{j}}{N_{n_j}})\right| \ge \epsilon_1.$$
This implies that 
$$
\left|\widehat{\nu_{>(n_j-1)}}(\frac12+k_{j})\right|\ge \epsilon_0\cdot\epsilon_1>0.
$$
Proposition \ref{prop_equiv-positive} (iii) holds, and thus $\{\nu_{>n}\}$ has an equi-positive subsequence. 
\end{proof}

\medskip

\begin{theorem}\label{lemma_bounded_B}
Let $N_n\ge 2$ be integer and let $B_n\subset\{0,1,...,N_n-1\}$. Suppose that $\liminf_{n\rightarrow\infty} \#B_n<\infty$. Then there exist $\epsilon_0>0$ and $\delta_0>0$ such that for all $n\ge 1$, we can find integer $k$ such that $\frac{k}{N_n}\in[0,1]\setminus\left(\frac12-\delta_0,\frac12+\delta_0\right)$ and 
$$
\left|\widehat{\delta_{B_n/N_n}}(\frac12+k)\right|\ge \epsilon_0.
$$
Hence, $\{\nu_{>n}\}$ is an equi-positive sequence.
\end{theorem}

\begin{proof}
Suppose that the conclusion is false. For any $j>2$, we can find $n_{j}$ such that for all integer $k$ with $\frac{k}{N_{n_j}}\not\in\left(\frac12-\frac1{2^j},\frac12+\frac1{2^j}\right)$, we have
\begin{equation}\label{eqsmall}
\left|\widehat{\delta_{B_{n_j}/N_{n_j}}}(\frac12+k)\right|<\frac1{2^j}.
\end{equation}
Denote $\nu_j$ to be the measure $\delta_{B_{n_j}/N_{n_j}}$. Since $\nu_{j}(\{1\}) = 0$, we can identify $\nu_j$ as a probability measure on ${\mathbb T}$.   Consider the complex measure 
$$
\nu_j'(E) = \int_E e^{-2\pi i \frac12x} d\nu_j(x).
$$
Then $\widehat{\nu_j'}(k) = \widehat{\nu_j}(\frac12+k)$. Moreover, $\nu_{j}'(\{b/N_{n_j}\}) = \frac{e^{-\pi i b/N_{n_j}}}{\#B_{n_j}}$ for $b\in B_{n_j}$ and these are all the atoms of $\nu_{j}'$.  Using Theorem \ref{discrete-inverse-theorem}, we have 
$$
\frac{1}{\#B_{n_j}} = \sum_{b\in B_{n_j}}\left|\nu_j'(\{b/N_{n_j}\})\right|^2 = \lim_{\ell\rightarrow \infty} \frac{1}{2 \ell N_{n_j}} \sum_{k= -\ell N_{n_j}}^{\ell N_{n_j}-1}\left |\widehat{\nu_j}(\frac12+k)\right|^2.
$$
As $\widehat{\nu_{j}}$ is $N_{n_j}$-periodic, we have 
\begin{equation}\label{eqfinitetight}
\begin{aligned}
\frac{1}{\#B_{n_j}} = &\lim_{\ell\rightarrow \infty} \frac{2\ell}{2 \ell N_{n_j}} \sum_{k= 0}^{ N_{n_j}-1} \left |\widehat{\nu_j}(\frac12+k)\right|^2\\
=&\frac{1}{ N_{n_j}}  \sum_{k= 0}^{ N_{n_j}-1} \left |\widehat{\nu_j}(\frac12+k)\right|^2\\
 = &\frac{1}{ N_{n_j}} \sum_{k/N_{n_j}\in\left(\frac12-\frac1{2^j},\frac12+\frac1{2^j}\right)} \left|\widehat{\nu_j}(\frac12+k)\right|^2 +\frac1{N_{n_j}}\sum_{k/N_{n_j}\not\in\left(\frac12-\frac1{2^j},\frac12+\frac1{2^j} \right)}\left|\widehat{\nu_j}(\frac12+k)\right|^2.\\
\end{aligned}
\end{equation}
As $|\widehat{\nu_j}(1/2+k)|\le 1$,
The first sum 
$$
\begin{aligned}
\frac{1}{ N_{n_j}} \sum_{k/N_{n_j}\in\left(\frac12-\frac1{2^j},\frac12+\frac1{2^j}\right)} |\widehat{\nu_j}(\frac12+k)|^2 \le& \frac{1}{N_{n_j}} \#\left\{\frac{k}{N_{n_j}}: \frac{k}{N_{n_j}}\in\left(\frac12-\frac1{2^j},\frac12+\frac1{2^j}\right)\right\}\\
\le& \frac{1}{N_{n_j}} \frac{1/2^{j-1}}{1/N_{n_j}} = \frac{1}{2^{j-1}}.\\
\end{aligned}
$$
Using (\ref{eqsmall}) and the fact that there are at most $N_{n_j}$ terms in the summation, the second sum 
$$
\frac1{N_{n_j}}\sum_{k/N_{n_j}\not\in\left(\frac12-\frac1{2^j},\frac12+\frac1{2^j} \right)}\left|\widehat{\nu_j}(\frac12+k)\right|^2 \le \frac{1}{2^j}.
$$
Combining with the fact that $\#B_{n_j}\le M$, we have 
$$
\frac{1}{M}\le \frac{1}{\#B_{n_j}} \le \frac{1}{2^{j-1}}+\frac{1}{2^{j}}.
$$
As $j$ can be arbitrarily large, the above cannot happen and we have a contradiction. This shows that our desired statement holds.
\end{proof}

\medskip

\begin{remark}
The key step of the proof of Theorem \ref{lemma_bounded_B} is to establish an identity
$$
\frac{1}{\#B_{n_j}} = \frac{1}{N_{n_j}}  \sum_{k= 0}^{ N_{n_j}-1} \left |\widehat{\nu_j}(\frac12+k)\right|^2.
$$
in (\ref{eqfinitetight}). This is actually equivalent to saying that $\{e^{2\pi i kx}: k = 0,1,...,N_{n_j}-1\}$ forms a tight Fourier frame for the measure $\nu_j = \delta_{B_{n_j}/N_{n_j}}$. The fact can also be deduced from finite frame theory (see e.g. \cite[Section 10]{DHL2017}). We leave it as an exercise for interested reader.
\end{remark}

\noindent{\bf Proof of Theorem \ref{maintheorem2}.}  Theorem \ref{maintheorem2} is a consequence of Theorem \ref{lemma_bounded_B} and Theorem \ref{theorem_admissible_spectral}\eproof

\medskip

 \section{Non-admissible Cantor-Moran measures (II): \ $\lim_{n\rightarrow\infty}\#B_{n} = \infty$.} After the previous section, the cases that we still cannot solve  are those $\#B_n$ that is not bounded on any subsequence.   Equivalently, $\lim_{n\rightarrow\infty}\#B_{n} = \infty$.  This also implies that $\lim\limits_{n\rightarrow\infty}N_{n} = \infty$. In this situation, we first show that  $\{\delta_{B_{n}/N_{n}}\}$ weakly converges to $\rho$.
 
 \medskip
 
 \begin{lemma}\label{lemma delta measure convergence}
Suppose that $\lim\limits_{n\rightarrow \infty}\#B_{n} = \infty$ and $\{\nu_{>n}\}$ weakly converges to $\rho$. Then $\{\delta_{B_{n}/N_{n}}\}$ weakly converges to $\rho$.
 \end{lemma}
\begin{proof}
To prove that $\{\delta_{B_{n}/N_{n}}\}$ weakly converges to $\rho$, we need to show that $\{\widehat{\delta_{B_{n}/N_{n}}}\}$ converges uniformly to $\widehat{\rho}$ on any compact subset $K\subset{\mathbb R}$. Note that
$$
\widehat{\nu_{>n}}(\xi) = \widehat{\delta_{B_{N_{n}}/N_{n}}} (\xi) \cdot \widehat{\nu_{>(n+1)}}\left(\frac{\xi}{N_{n}}\right).
$$
Using the equality, one can get that
\begin{eqnarray}\label{delta_weak_converge}
\left|\widehat{\delta_{B_{N_{n}}/N_{n}}} (\xi) - \widehat{\rho}(\xi)\right|
&\le & \left|\widehat{\delta_{B_{N_{n}}/N_{n}}} (\xi) - \widehat{\nu_{>n}}(\xi)\right| + \left|\widehat{\nu_{>n}}(\xi) - \widehat{\rho}(\xi)\right| \\
&\le & \left|1 - \widehat{\nu_{>(n+1)}}\left(\frac{\xi}{N_{n}}\right)\right| + \left|\widehat{\nu_{>n}}(\xi) - \widehat{\rho}(\xi)\right|. \nonumber
\end{eqnarray}
Since $\{\nu_{>n}\}$ weakly converges to $\rho$, for any $\epsilon >0$, there exists $M_1>0$ such that for all $x\in K$, we have
\begin{eqnarray}\label{inequ_1}
\left|\widehat{\nu_{>n}}(\xi) - \widehat{\rho}(\xi)\right| < \epsilon
\end{eqnarray}
whenever $n>M_1$. The equicontinuity of $\{\widehat{\nu_{>(n+1)}}\}$ implies that for the above $\epsilon >0$, there is $\delta >0$ such that $|\widehat{\nu_{>(n+1)}}(x) - \widehat{\nu_{>(n+1)}}(y)| < \epsilon$ whenever $|x-y|<\delta$. Note that $\widehat{\nu_{>(n+1)}}(0) = 1$, we can take $n> M_2$ for some $M_2>0$ so that $|\frac{x}{N_{n}}|<\delta$ for any $x\in K$. Therefore, one can get 
\begin{eqnarray}\label{inequ_2}
\left|\widehat{\nu_{>(n+1)}}\left(\frac{\xi}{N_{n}}\right) - 1\right| < \epsilon
\end{eqnarray}
whenever $n>M_2$. Take $M=\max\{M_1, M_2\}$, if $n>M$, then (\ref{inequ_1}) and (\ref{inequ_2}) hold simultaneously for all $x\in K$. Substituting them into (\ref{delta_weak_converge}), one can get the uniform convergence of $\{\widehat{\delta_{B_{N_{n}}/N_{n}}}\}$ to $\widehat{\rho}$ on $K$. Thus our statement follows.

\end{proof} 
 
 \medskip
 
Not only the discrete measure  weakly converges  to $\rho$, Lemma \ref{lemma_discrete} is also now a necessary and sufficient condition. We first prove the following lemma.

\begin{lemma}\label{lemma_iterate}
Let $k\in{\mathbb Z}$ be written as
$$
k = \ell_1+ N_{n+1}\ell_2+N_{n+1}N_{n+2}\ell_3+...+N_{n+1}N_{n+2}...N_{n+r-1}\ell_r,
$$
where $\ell_t\in \{0, 1, \cdots, N_{n+t-1}-1\}$ for $t=1,...,r,r+1$. Then define the sequence
$$
 \xi_{1} = \frac12,  \ \xi_{t} = \frac{\xi_{t-1}+\ell_{t-1}}{N_{n+(t-1)}},
 $$
we have 
\begin{equation}\label{eq_iterate}
\widehat{\nu_{>n}}(\frac12+k)  =  \left(\prod_{t=1}^{r} \widehat{\delta_{B_{n+t}/N_{n+t}}}(\xi_{t}+\ell_{t})\right)\cdot\widehat{\nu_{>(n+r)}} (\xi_{r+1})  .
    \end{equation}
\end{lemma}

\begin{proof}
We use the identity $\nu_{>n}(\xi) = \widehat{\delta_{B_{n+1}/N_{n+1}}}(\xi)\cdot \widehat{\nu_{>(n+1)}}(\frac{\xi}{N_{n+1}})$. Note that $\widehat{\delta_{B_{n+1}/N_{n+1}}}$ is $N_{n+1}$-periodic and 
$$
\frac{\frac12+k}{N_{n+1}} = \xi_2+\ell_2+N_{n+2}\ell_3+....+N_{n+2}....N_{n+r-1}\ell_r.
$$
We have
$$
\widehat{\nu_{>n}}(\frac12+k) = \widehat{\delta_{B_{n+1}/N_{n+1}}}(\xi_{1}+\ell_{1}) \cdot \widehat{\nu_{>(n+1)}}(\xi_2+\ell_2+N_{n+2}\ell_3+....+N_{n+2}....N_{n+r-1}\ell_r).
$$
We iterate the formula and note that $\widehat{\delta_{B_{n+t}/N_{n+t}}}$ is $N_{n+t}$-periodic.  (\ref{eq_iterate}) follows.
\end{proof}

 \medskip

 \begin{theorem}\label{theorem_6.3}
 Suppose that $\{\nu_{>n}\}$ converges weakly to $\rho$. The following are equivalent.
 \begin{enumerate}
     \item there exist $\epsilon_0>0$, $\delta_0>0$ and  a subsequence $\{n_j\}$ such that for all $n_j$, we can find integer $k_{j}$ such that $\frac{k_{j}}{N_{n_j}}\in[0,1]\setminus\left(\frac12-\delta_0,\frac12+\delta_0\right)$ and 
$$
\left|\widehat{\delta_{B_{n_j}/N_{n_j}}}(\frac12+k_{j})\right|\ge \epsilon_0.
$$
\item there exist $\epsilon_0>0$ and a subsequence $\{\nu_{>n_j}\}$ such that  for any $j\ge 1$, we can find integer $k_{j}$ such that
    $$\left|\widehat{\nu_{>n_j}}(\frac12+k_{j})\right|\ge \epsilon_0.$$
 \end{enumerate}
 \end{theorem}
 \begin{proof}
 ((i)$\Rightarrow$(ii)) follows from Lemma \ref{lemma_discrete}. We now suppose (ii) holds. Then we can find $k_{j}$ such that 
 \begin{eqnarray}\label{inequality_epsilon}
\left|\widehat{\nu_{>n_j}}(\frac12+k_{j})\right|\ge \epsilon_0. 
 \end{eqnarray}
As we know that $|\overline{\widehat{\nu}(\xi)}| = |\widehat{\nu}(-\xi)|$, there is no loss of generality to assume $k_{j}\geq0$. Then we can write $k_{j}$ as follows:
 $$
 k_{j} = \ell_{1,j}+ N_{n_j+1}\ell_{2,j}+N_{n_j+1}N_{n_j+2}\ell_{3,j}+...+N_{n_j+1}N_{n_j+2}..N_{n_j+r_j}\ell_{r_j+1,j},
 $$
 where $\ell_{t,j}\in\{0,1,..., N_{n_j+t-1}-1\}$ for all $t=1,...,r_j,r_j+1$. Define
 $$
 \xi_{1,j} = \frac12,  \ \xi_{t,j} = \frac{\xi_{t-1,j}+\ell_{t-1,j}}{N_{n_j+(t-1)}}.
 $$
By Lemma \ref{lemma_iterate}, we have 
 $$
\begin{aligned}
\epsilon_0\le \left|\widehat{\nu_{>n_j}}(\frac12+k_j)\right| & = & \left(\prod_{t=1}^{r_j} |\widehat{\delta_{B_{n_j+t}/N_{n_j+t}}}(\xi_{t,j}+\ell_{t,j})|\right)\cdot \left|\widehat{\nu_{>(n_j+r_j)}} (\xi_{r_{j}+1,j})  \right|.
\end{aligned} 
 $$
 As all Fourier transforms are less than 1 in modulus, this implies that for all $t=1,...,r_j$,
 \begin{equation}\label{eq5.03}
\left|\widehat{\delta_{B_{n_j+t}/N_{n_j+t}}}(\xi_{t,j}+\ell_{t,j})\right|\ge \epsilon_0 \ \mbox{and} \  \left|\widehat{\nu_{>(n_j+r_j)}} (\xi_{r_{j}+1,j})  \right|\ge \epsilon_0.
      \end{equation}
By the equicontinuity of the family of measures  in $\widehat{{\mathcal P}}[0,1]$, we can always find a $\delta_0>0$ independent of the measure such that whenever $|x-y|\le \delta_0$, for all $\nu\in {\mathcal P}[0,1]$,
$$
\left|\widehat{\nu}(x)-\widehat{\nu}(y)\right|<\frac{\epsilon_0}{2}.
$$
In particular, it holds for $\delta_{B_{n_j+t}/N_{n_j+t}}$ and $\nu_{>(n_j+r_j)}$.


\vspace{0.2cm}

\noindent {\bf{Claim:}} There exist infinitely many $j$  such that  we can find $t\in\{1,...,r_{j},r_{j+1}\}$ satisfying $$\frac{\ell_{t,j}}{N_{n_j+t}}\not\in\left(\frac12-\frac{\delta_0}{2},\frac12+\frac{\delta_0}{2}\right).$$

\medskip

Suppose the claim is false. Then for all $j$ large enough and any $t\in\{1,...,r_{j},r_{j+1}\}$,
$$
\frac{\ell_{t,j}}{N_{n_j+t}}\in\left(\frac12-\frac{\delta_0}{2},\frac12+\frac{\delta_0}{2}\right).
$$
We now estimate the distance between $1/2$ and $\xi_{t,j}$ for all $t=2,...,r_j+1,$
\begin{equation}\label{eq5.04}
\begin{aligned}
\left|\xi_{t,j}-\frac12\right| = &\left|\frac{\xi_{t-1,j}}{N_{n_j+t-1}}+\frac{\ell_{t-1,j}}{N_{n_j+t-1}}-\frac12\right|\\
\le& \frac{\left|\xi_{t-1,j}-\frac12\right|}{N_{n_j+t-1}} +\frac{1}{2N_{n_j+t-1}} +\frac{\delta_0}{2}\\
\le&\frac{\left|\xi_{t-2,j}-\frac12\right|}{N_{n_j+t-2}N_{n_j+t-1}} +\frac12\left(\frac{1}{N_{n_j+t-2}N_{n_j+t-1}} +\frac{1}{N_{n_j+t-1}}\right)+\frac{\delta_0}{2}\left(1+\frac{1}{N_{n_j+t-1}}\right)\\
\le &  \ \ \vdots\\
\le &\frac12\left(\frac{1}{N_{n_j+1}...N_{n_j+t-2}N_{n_j+t-1}}+...+\frac{1}{N_{n_j+t-2}N_{n_j+t-1}}+\frac{1}{N_{n_j+t-1}}\right) \\
& \ \ +\frac{\delta_0}{2}\left(1+\frac{1}{N_{n_j+t-1}}+...+\frac{1}{N_{n_j+t-1}...N_{n_j+1}}\right)\\
\le & \frac12\cdot \frac{2}{N_{n_j+t-1}} +\frac{\delta_0}{2}\left(1+\frac2{N_{n_j+t-1}}\right).
\end{aligned}
\end{equation}
Recall that in this section, we have $\lim_{n\to\infty}N_n=\infty$. The above inequality tells us that when $j$ is large enough, we can have $\left|\xi_{t,j}-\frac12\right|<\delta_0$.
Hence, $\xi_{r_j+1,j}\in (1/2-\delta_0,1/2+\delta_0)$. By the equicontinuity,
$$
\epsilon_0\le\left|\widehat{\nu_{>(n_j+r_j)}} (\xi_{r_{j+1},j})  \right|\le \left|\widehat{\nu_{>(n_j+r_j)}} (\frac12)  \right|+\epsilon_0/2.
$$
But $\{\nu_{>n}\}$ weakly converges to $\rho$, $\{\widehat{\nu_{>(n_j+r_j)}} (\frac12)  \}$ converges to $0$ as $j$ goes to infinity. This leads to a contradiction to (\ref{eq5.03}) when $j$ is large. This justifies the claim.

\medskip

Having the claim,  we then  take $t$ to be the first integer such that the claim holds and denote it by $t_j$. i.e.
$$
\frac{\ell_{t_j,j}}{N_{n_j+t_j}}\not\in\left(\frac12-\frac{\delta_0}{2},\frac12+\frac{\delta_0}{2}\right).
$$
We also know that $|\xi_{t_j,j}-\frac12|<\delta_0$ using the same estimation argument in (\ref{eq5.04}). Hence, equicontinuity implies that $|\widehat{\delta_{B_{n_j+t_j}/N_{n_j+t_j}}}(1/2+\ell_{t_j,j})|\ge \epsilon_0/2$. (i) follows by taking $\epsilon_0/2$, $\delta_0$ in the equicontinuity and the subsequence $n_j+t_j$. This completes the proof.

\end{proof}

 

 \medskip
 
 Because of the Lemma \ref{lemma delta measure convergence} and the equivalent conditions of weak convergence in Lemma \ref{lem_weak_equivalence}, for all $\epsilon>0$ and $\delta>0$, we can find $n_0$ such that whenever $n\ge n_0$, 
 $$
  \frac{\#\{b/N_n\in B_n/N_n: b/N_n\in(\delta,1-\delta)\}}{\#B_n}= \delta_{B_n/N_n}(\delta,1-\delta)< \epsilon.
 $$
 Let $B_{n,\delta} = B_n\setminus\{b\in B_n: b/N_n\in(\delta,1-\delta)\}$. Whenever $n\ge n_0$, we have
 \begin{equation}\label{eq5.5}
  \widehat{\delta_{B_n/N_n}}(\xi) = \frac{1}{\#B_n} \sum_{b\in B_{n,\delta}}e^{-2\pi i b\xi/N_n}+\frac{1}{\#B_n} \sum_{b\in B_n\setminus B_{n,\delta}}e^{-2\pi i b\xi/N_n}.  
 \end{equation}
 But 
 $$
\left| \frac{1}{\#B_n} \sum_{b\in B_n\setminus B_{n,\delta}}e^{-2\pi i b\xi/N_n}\right|\le  \frac{\#\{b/N_n: b/N_n\in(\delta,1-\delta)\}}{\#B_n} <\epsilon.
 $$
 By (\ref{eq5.5}), we have
 \begin{equation}\label{eq_6.8+}
\left| \widehat{\delta_{B_n/N_n}}(\xi)\right| \ge \left|\frac{1}{\#B_n} \sum_{b\in B_{n,\delta}}e^{-2\pi i b\xi/N_n}\right| -\epsilon.
\end{equation}
Let 
$$
B_{n,\delta,0} = B_{n,\delta}\cap [0,N_n\delta], \   \ B_{n,\delta,1} = N_n- (B_{n,\delta}\cap N_n[1-\delta,1)).$$
We have now
$$
\frac{1}{\#B_{n}} \sum_{b\in B_{n,\delta, 1}}e^{-2\pi i (N_n-b)\xi/N_n} = \frac{1}{\#B_{n}} \sum_{b\in B_{n,\delta,1}}e^{-2\pi i (1-\frac{b}{N_n})\xi},
$$
and
$$
\begin{aligned}
\frac{1}{\#B_n} \sum_{b\in B_{n,\delta}}e^{-2\pi i \frac{b}{N_n}\xi} 
= & \frac{1}{\#B_{n}}\left( \sum_{b\in B_{n,\delta,0}}e^{-2\pi i \frac{b}{N_n}\xi}
+e^{-2\pi i \xi}\sum_{b\in B_{n,\delta,1}}e^{2\pi i \frac{b}{N_n}\xi}\right).
\end{aligned}
$$
 Putting $\xi = \frac12+k$, we have 
\begin{equation}\label{eq_6.9}
\begin{aligned}
\left|  \sum_{b\in B_{n,\delta}}e^{-2\pi i \frac{b}{N_n}(\frac12+k)}\right|^2 =& \left|  \sum_{b\in B_{n,\delta,0}}e^{-2\pi i \left(\frac{b}{2N_n}+\frac{b}{N_n}k\right)}-\sum_{b\in B_{n,\delta,1}}e^{2\pi i \left(\frac{b}{2N_n}+\frac{b}{N_n}k\right)}\right|^2\\
=&\left|\sum_{b\in B_{n,\delta,0}}\cos 2\pi\left(\frac{b}{2N_n}+\frac{b}{N_n}k\right)-\sum_{b\in B_{n,\delta,1}}\cos 2\pi\left(\frac{b}{2N_n}+\frac{b}{N_n}k\right)\right|^2\\
& \ +\left|\left(\sum_{b\in B_{n,\delta,0}} +\sum_{b\in B_{n,\delta,1}}\right)\sin 2\pi\left(\frac{b}{2N_n}+\frac{b}{N_n}k\right)\right|^2.
 \end{aligned}
\end{equation}
We say that $(N_n,B_n)$ is {\it symmetric} if for all $\delta>0$, we can find some $n_0$ such that for all $n\ge n_0$, $B_n = B_{n,\delta}$ and  $B_{n,\delta,1} = B_{n,\delta,0}\setminus\{0\}$. From the above derivation, we have the following lemma. 

\begin{lemma}\label{lemma_estimate}
Suppose that $\{\nu_{>n}\}$ converges weakly to $\rho$ and $\lim_{n\rightarrow\infty}\#B_n=\infty$. Then for all $\epsilon>0$ and for all $\delta>0$, we can find $n_0$ such that for all $n\ge n_0$, we have
\begin{equation}\label{eq_6.10}
\left|\widehat{\delta_{B_n/N_n}}(\frac12+k)\right| \ge \left|\frac{1}{\#B_{n}}\sum_{b\in B_{n,\delta,0}\cup B_{n,\delta,1}} m_b \sin 2\pi\left(\frac{b}{2N_n}+\frac{b}{N_n}k\right)\right|-\epsilon,
 \end{equation}
 where $m_b = 2$ if $b\in B_{n,\delta,0}\cap B_{n,\delta,1}$ and $m_b = 1$ otherwise. Furthermore,  suppose that $(N_n,B_n)$ is symmetric. Then for all $\delta>0$, we can find $n_0$ such that whenever $n\ge n_0$,
\begin{eqnarray}\label{eq_symmetric}
 \left|\widehat{\delta_{B_n/N_n}}(\frac12+k)\right|^2 =\frac{1}{(\#B_n)^2} +\left|\frac{2}{\#B_{n}}\sum_{b\in B_{n,\delta,0}}\sin 2\pi\left(\frac{b}{2N_n}+\frac{b}{N_n}k\right)\right|^2. 
\end{eqnarray}
\end{lemma} 
 
\begin{proof}
(\ref{eq_6.10}) follows directly from (\ref{eq_6.8+}) and (\ref{eq_6.9}) by dropping the first term of the cosine in (\ref{eq_6.9}). For (\ref{eq_symmetric}), we note that since $B_{n} = B_{n,\delta}$. There is no $\epsilon$ loss in (\ref{eq_6.8+}). Finally, as $B_{n,\delta,1} = B_{n, \delta,0}\setminus\{0\}$, the cosine term in (\ref{eq_6.9}) are cancelled out except the case $b = 0$, which equals $1$. Also, $m_b = 2$ for all $b\ne 0$ because of the symmetry. This shows $(\ref{eq_symmetric})$.
\end{proof}

\vspace{0.2cm}

\section{Non-admissible Cantor-Moran measures (III): A sum of sine problem}
\subsection{Symmetric case.} 
In this section, we will prove Theorem \ref{maintheorem3} and Theorem \ref{maintheorem4}. For any real number $\xi$, we denote $\langle\xi\rangle$ to be the unique number suth that $|\langle\xi\rangle|\le\frac12$ and $[\xi] :=\xi-\langle\xi\rangle\in{\mathbb Z}$, which denotes the integer part of $\xi$.

\medskip

\noindent{\bf Proof of Theorem \ref{maintheorem3}.} It suffices to show that the first statement of Theorem \ref{theorem_6.3} holds.
Given $\epsilon_0$ and $\delta_0$ in the sum of sine problem which holds for $\{B_{n_j,\delta,0}\}$. We take $0<\delta< \frac{\epsilon_0}{4\pi}$. Consider $n_j$ large enough so that  $1/N_{n_j}<\delta/2$ and (\ref{eq_symmetric}) holds. From the sum of sine problem, we can find $x_j\in[0,1/2-\delta_0]$ such that 
$$
\left|\frac{2}{\#B_{n_j}}\sum_{b\in B_{n_j,\delta,0}}\sin \left(2\pi b x_j\right)\right|\ge \epsilon_0.
$$
We now take $k_j = [N_{n_j} x_j]$. Let $\xi_b = \frac{b}{2N_{n_j}} - \frac{b\langle N_{n_j}x_j\rangle}{N_{n_j}}$. From (\ref{eq_symmetric}),  we have
\begin{equation}\label{eq7.1}
\begin{aligned}
\left|\widehat{\delta_{B_{n_j}/N_{n_j}}}(\frac12+k_j)\right|\ge \left|\frac{2}{\#B_{n_j}}\sum_{b\in B_{n_j,\delta,0}}\sin 2\pi bx_j\right|- \frac{2}{\#B_n} \sum_{b\in B_{n_j,\delta,0}} \left|\sin 2\pi\left(\xi_b+bx_j\right)-\sin 2\pi bx_j \right|.
\end{aligned}
\end{equation}
By the standard sum-to-product trigonometric identity,
\begin{equation}\label{eq_trig}
    \sin 2\pi\left(\xi_b+bx_j\right)-\sin 2\pi bx_j = 2\sin\left(\pi\xi_b\right)\cos\left(\pi (\xi_b+2b x_j)\right).
\end{equation}
 Note that $|\frac{b}{N_{n_j}}|\le \delta$ and $|\langle N_{n_j}x_j\rangle|\le 1/2$, we have $|\xi_b|\le \delta$. This implies that 
$$
\frac{2}{\#B_{n_j}} \sum_{b\in B_{n_j,\delta,0}} \left|\sin 2\pi\left(\xi_b+bx_j\right)-\sin 2\pi bx_j \right|\le \frac{4\pi }{\#B_{n_j}} \sum_{b\in B_{n_j,\delta,0}}|\xi_b| \le 2\pi \delta.
$$
Combining it with (\ref{eq7.1}), we obtain that 
$$
\left|\widehat{\delta_{B_{n_j}/N_{n_j}}}(\frac12+k_j)\right|\ge\epsilon_0-2\pi \delta \ge \epsilon_0/2.
$$
As $x_j\in[0,\frac12-\delta_0]$, we have $k_j/N_{n_j}<1/2-\delta_0$. This shows that the first statement of Theorem \ref{theorem_6.3} holds, as desired.
\eproof

\bigskip

We now use the Bourgain's example  of small sum of sine to prove Theorem \ref{maintheorem4}. Let us recall the precise statement below.

\bigskip

\begin{theorem}
Given a positive integer $n$, one can choose a set of positive integers $B_{n}' = \{b_1<b_2<...<b_n\}$ such that 
$$
\max_{x\in[0,1]}\left|\sum_{b\in B_n'}\sin (2\pi b x)\right|\le Cn^{2/3}
$$
where $C$ is an absolute constant.
\end{theorem}

\medskip

\noindent{\bf Proof of Theorem \ref{maintheorem4}.} (i) Using the Bourgain's example. For all positive integer $n$, we can find some finite sets of integers $B_{n}'$ of cardinality $n$ such that 
$$
\max_{x\in\left[0,1\right]}\left|\frac{1}{n}\sum_{b\in B_{n}'} \sin (2\pi b x)\right| \le C\frac{n^{2/3}}{n}  = \frac{C}{n^{1/3}}.
$$
 We take  $N_n > (\max B_n') 2^n$  and  $B_n = B_n'\cup (N_n- B_n'\setminus\{0\})$. Then $(N_n,B_n)$ is symmetric. We now consider the Cantor-Moran measure $\mu = \mu(N_n,B_n)$. Note that by our choice of $N_n$,  for any $\eta>0$, the Dirac measure $\delta_{B_n/N_n}(\eta,1-\eta) =0$ for all $n$ large.  This also implies that the associated measure $\nu_{>n}$ has no support in $(\eta,1-\eta)$ as long as $n$ is sufficiently large. Hence, $\{\nu_{>n}\}$ converges weakly to $\rho$. Finally, $\lim_{n\rightarrow\infty}\#B_n =\infty$  by our construction. Our proof will be complete if we can show that statement (i) in Theorem \ref{theorem_6.3} does not hold. 
 
 \bigskip
 It suffices to show that 
 \begin{equation}\label{eq_lim7.1}
 \lim_{n\rightarrow\infty}\max_{k/N_n\in[0,1]}\left|\widehat{\delta_{B_n/N_n}}(\frac12+k)\right| = 0.
      \end{equation}
 To see this, we let $x = k/N_n$. Note that 
 $$
 \begin{aligned}
 \left|\frac{2}{\#B_{n}'}\sum_{b\in B_{n}'}\sin 2\pi\left(\frac{b}{2N_n}+\frac{b}{N_n}k\right)\right| \le &  \left|\frac{2}{\#B_{n}'}\sum_{b\in B_{n}'}\sin (2\pi bx)\right| +\frac{2}{\#B_{n}}\sum_{b\in B_{n}'}\left|\sin 2\pi\left(\frac{b}{2N_n}+ bx\right)-\sin (2\pi bx)\right|\\
 \le& \frac{2C}{n^{1/3}}+\frac{4\pi}{\#B_n'} \sum_{b\in B_n'}\left|\frac{b}{2N_n}\right| \   \ \ \ \ \ (\mbox{using the trig identity \eqref{eq_trig}} )\\
  \le &\frac{2C}{n^{1/3}}+\frac{\pi}{2^{n-1}} \ \ \ \   \  \ \left(\mbox{since}  \ \frac{b}{N_n}\le \frac{\max B_n'}{N_n}<\frac1{2^n}.\right)\\
 \end{aligned}
 $$
 Using Lemma \ref{lemma_estimate}, we have 
 $$
 \begin{aligned}
  \max_{k/N_n\in[0,1]}\left|\widehat{\delta_{B_n/N_n}}(\frac12+k)\right|^2  \le &\frac{1}{(\#B_n')^2} +\left(\frac{2C}{n^{1/3}}+\frac{\pi}{2^{n-1}}\right)^2.
 \end{aligned}
 $$
The right hand side goes to zero as $n$ tends to infinity. This establishes (\ref{eq_lim7.1}).  The proof is complete. 

\medskip

(ii) Note that the Cantor-Moran measure constructed in (i) does not have any equi-positive subsequence. Using Proposition \ref{prop_equiv-positive} (ii), we know that $\{\widehat{\nu_{>n}}\}$ converges uniformly to $\widehat{\rho}$ on $\frac12+\Z$. 
\eproof

\medskip
\subsection{Uniform convergence of Fourier transform on non-compact sets} From Lemma \ref{lem_weak_equivalence}, we know that  a sequence of probability measures $\{\nu_n\}$ converges weakly to a probability measure $\nu$ if and only if $\{\widehat{\nu_n}\}$ converges uniformly to $\widehat{\nu}$ on all compact subsets of ${\mathbb R}$. A problem of independent interest is to ask if we can have uniform convergence over non-compact sets. 
The following simple example shows that it is possible if we do not restrict our measures $\nu_n$ to be supported on $[0,1]$

  \begin{example}
  Let $\psi$ be a non-negative compactly supported smooth function on $[0,1]$. Define $\psi_{n}(x)  = n\psi(nx)$ . Then the absolutely continuous  measure $\{\nu_n\}$ with density $ \rho\ast\psi_{n}$   converges weakly to $\rho$.  However, 
  $$
  \widehat{\nu_n}(\xi) = \widehat{\rho}(\xi) \widehat{\psi_{n}}(\xi).
  $$
  Hence, as $\widehat{\rho}(\frac12+k) = 0$, we must have  $\widehat{\nu_n}(\frac12+k) = 0$ for all $k\in\Z$. This implies that $\{\widehat{\nu_n}\}$ converges uniformly to $\widehat{\rho}$ on the non-compact set $\frac12+\Z$.   
   \end{example}

Note that the support of $\nu_n$ in the above example is inside $[0,\frac1n]\cup[1,1+\frac1n]$. However, if we restrict our attention to $\nu_n$ being supported only inside $[0,1]$. The problem becomes much harder. Yet,  Theorem \ref{maintheorem4} (ii) tells us that it is possible to converge uniformly on  $\frac12+\Z$! The following proposition shows that uniform convergence on unbounded set is not easy to achieve.  $\frac12$ is the only special point that can  allow uniform convergence.

\begin{proposition}
Suppose that $\{\nu_{n}\}$ is a sequence of probability measures on $[0,1]$ that converges weakly to $\rho$ and $\nu_{n}(\{1\}) = 0$. Then for any $\xi\in[0,1]$ and $\xi\ne \frac12$, $\{\widehat{\nu_n}\}$ does not converge uniformly to $\widehat{\rho}$ on $\xi+\Z$.
\end{proposition}

\begin{proof}
Suppose the conclusion is false. We will have $\{\widehat{\nu_{n}}(\xi+k)\}$ converges uniformly to $\widehat{\rho}(\xi+k)$ for all $k\in\Z$.  We now identify $\nu_{n}$ as a measure on ${\mathbb T}$. As $\nu_{n}$ has no atom, we can define the complex measure 
$$
\nu_{n,\xi} (E)  =  \int_E e^{-2\pi i \xi x}d\nu_{n}(x).
$$
This measure has no atom on ${\mathbb T}$ either. Hence, by Theorem \ref{discrete-inverse-theorem}, we have 
\begin{equation}\label{eq5.3}
\lim_{N\rightarrow\infty} \frac{1}{2N+1} \sum_{k= -N}^{N}|\widehat{\nu_{n,\xi}}(k)|^2=0.
    \end{equation}
Note that 
$$
\widehat{\nu_{n,\xi}}(k) = \int e^{-2\pi i kx} e^{-2\pi i \xi x}d\nu_{n}(x) = \widehat{\nu_{n}}(\xi+k),
$$
which converges to $\widehat{\rho}(\xi+k)$ uniformly for $k\in\Z$. Recall that $\widehat{\rho}(\xi) = e^{\pi i \xi}\cos(\pi \xi)$. We now claim that  $\{|\widehat{\nu_{n}}(\xi+k)|^2\} $ also converges to $\left|\widehat{\rho}(\xi)\right|^2$ uniformly for $k\in{\mathbb Z}$. Indeed, it follows from the following estimation: 
$$
\begin{aligned}
\left||\widehat{\nu_{n}}(\xi+k) |^2-\left|\widehat{\rho}(\xi)\right|^2\right| 
=& \big||\widehat{\nu_{n}}(\xi+k)| -\left|\widehat{\rho}(\xi)\right|\big|\cdot\big|\left|\widehat{\nu_{n}}(\xi+k)\right| +\left|\widehat{\rho}(\xi)\right|\big|\\
\le & 2 \big|\left|\widehat{\nu_{n}}(\xi+k)\right| -\left|\widehat{\rho}(\xi)\right|\big|\\
= &2 \big|\left|\widehat{\nu_{n}}(\xi+k)\right| -\left|\widehat{\rho}(\xi+k)\right|\big|\\
\le &  2 \big|\widehat{\nu_{n}}(\xi+k) -\widehat{\rho}(\xi+k)\big|.
\end{aligned}
$$
As the right hand side converges uniformly for $k\in\Z$, $\{|\widehat{\nu_{n}}(\xi+k) |^2\} $ converges to $\left|\widehat{\rho}(\xi)\right|^2$ uniformly for $k\in \Z$. Now, given any $\epsilon>0$, one can find $n$ such that $\left||\widehat{\nu_{n}}(\xi+k) |^2-\left|\widehat{\rho}(\xi)\right|^2\right|<\frac{\epsilon}{2}$ for all $k\in\Z$ and (\ref{eq5.3}) tells us that we can find $N$ large so that 
$$
\frac{1}{2N+1} \sum_{k= -N}^{N}|\widehat{\nu_{n}}(\xi+k) |^2<\frac{\epsilon}{2}.
$$
Therefore,
$$
\begin{aligned}
\left|\widehat{\rho}(\xi)\right|^2\le& \left|\frac{1}{2N+1}\sum_{k= -N}^{N}|\widehat{\nu_{n}}(\xi+k) |^2-\left|\widehat{\rho}(\xi)\right|^2\right|+\frac{1}{2N+1}\sum_{k= -N}^{N}|\widehat{\nu_{n}}(\xi+k) |^2\\
=& \left|\frac{1}{2N+1}\sum_{k= -N}^{N}\left(|\widehat{\nu_{n}}(\xi+k) |^2-\left|\widehat{\rho}(\xi)\right|^2\right)\right|+\frac{1}{2N+1}\sum_{k= -N}^{N}|\widehat{\nu_{n}}(\xi+k) |^2\\
<&\epsilon.\\
\end{aligned}
$$
As $\epsilon$ is arbitrary, this forces $\left|\widehat{\rho}(\xi)\right|^2 = 0$. However, this is impossible since $\xi \ne \frac12$. Hence, there cannot be a uniform convergence of  $\{\widehat{\nu_{n}}\}$  to  $\widehat{\rho}$ on $\xi+\Z$, completing the proof. 
\end{proof}

\medskip

\subsection{General non-symmetric case.} If $(N_n,B_n)$ is not symmetric, one would require the following asymmetric version of the sum of sine conjecture. 

\medskip

\noindent{\bf Asymmetric sum of sine problem:} Let ${\mathcal B}$ be a collection of finite sets of positive integers and for each $B\in{\mathcal B}$, we associate a weight $M(B) = (m_b)_{b\in B}$ with $m_b = 1$ or $2$. We say that the {\it asymmetric sum of sine problem holds for ${\mathcal B}$ with weight $M(B)$} if  there exists $\epsilon_0>0$ and $\delta_0>0$ such that for any finite set of positive integers $B\in{\mathcal B}$ 
$$
\max_{x\in\left[0,\frac12-\delta_0\right]}\left|\sum_{b\in B} m_b \sin (2\pi b x)\right| \ge \epsilon_0 (\#B),
$$

\medskip

We have the following theorem.

\begin{theorem}
Let $\{(N_n,B_n)\}$ be the sequence that generates Cantor-Moran measure $\mu(N_n,B_n)$ with $\{\nu_{>n}\}$ converging weakly to $\rho$ and $\lim_{n\rightarrow\infty}\#B_n = \infty$. Suppose that the asymmetric sum of sine conjecture holds for ${\mathcal B} = \{B_{n,\delta,0}\cup B_{n,\delta,1}\}$ with weight $m_b = 2$ if $b\in B_{n,\delta,0}\cap B_{n,\delta,1}$ and $m_b = 1$ otherwise. Then $\{\nu_{>n}\}$ has an equi-positive subsequence and $\mu(N_n,B_n)$ is a frame-spectral measure if each $(N_n,B_n,L_n)$ forms a frame triple.
\end{theorem}

The proof is identical to the proof of Theorem \ref{maintheorem3} using (\ref{eq_6.10}) instead of (\ref{eq_symmetric}). We will omit the detail.

\medskip

\section{Examples of  sum of sine problem}
  Although the Bourgain's sum of sine problem cannot hold for the class of all finite set of integers, we can still verify several large subclasses such that the sum of sine problem holds. They are in Proposition \ref{exam-sine-conjecture} below. It means that whenever we take $B_n$ from the following subclasses or finitely many of the following subclasses, the Cantor-Moran measure $\mu(N_n,B_n)$ will have an equi-positive subsequence $\{\nu_{>n}\}$ and hence, spectrality or frame-spectrality problem can be solved. 

\medskip

\begin{proposition}\label{exam-sine-conjecture}
The sum of sine problem is true for the following class of finite sets of positive integers.
\begin{enumerate}
    \item ${\mathcal B}_M = \{ B\subset{\mathbb Z}^+: \#B\le M \}$, where $M$ is a positive constant.
    \item ${\mathcal C}_c  = \{B\subset{\mathbb Z}^+: \frac{\#B}{\max B}\ge c \}$, where $c$ is a positive constant.
    \item ${\mathcal L}_A  = \{B\subset{\mathbb Z}^+: B \ \mbox{is lacunary with lacunary constant A} \}$ and $A>2$. 
\end{enumerate}
\end{proposition}

\medskip

\subsection{$\#B$ is Uniformly bounded}
The following Turan theorem resembles the sum of sine problem and it will be used in the proof of Proposition \ref{exam-sine-conjecture}. This theorem is also commonly referred as the {\it Turan-Nazarov inequality}. 

\begin{theorem} (\cite[Theorem 1.4]{N94})\label{TZ-ineq}
Let $p(z) = \sum_{k=1}^n c_k z^{m_k}$ where $c_k\in{\mathbb C}$ and $m_1<...<m_n\in{\mathbb Z}$ be a trigonmetric polynomial on the unit circle ${\mathbb T}$, and let $E$ be a measurable subset of ${\mathbb T}$. Then
\begin{equation}\label{TZ-inequality}
\sum_{k=1}^n|c_k|\le \left(\frac{14}{|E|}\right)^{n-1} \sup_{z\in E}|p(z)|.
\end{equation}
(Here, $|E|$ denotes the Lebesgue measure of $E$.)
\end{theorem}

\medskip
 
 \noindent \textbf{Proof of Proposition \ref{exam-sine-conjecture} (i):}
 For any $B\in {\mathcal B}_M$, it determines a trigonmetric polynomial $p(z)=\sum_{b\in B}z^b-\sum_{b\in B}z^{-b}$ on $\mathbb{T}$.
 From Theorem \ref{TZ-ineq} and taking $E = [0,1/3]$, we have 
 $$\max_{x\in [0, \frac13]}\left|p(e^{2\pi i x})\right|\geq \left(\frac1{42}\right)^{2\#B-1}\cdot2(\#B)\geq2\left(\frac1{42}\right)^{2M-1}\cdot(\# B).$$
  On the other hand,  
 \begin{eqnarray*}
p\left(e^{2\pi i x}\right)&=&\sum_{b\in B}e^{2\pi i bx}-\sum_{b\in B}e^{-2\pi i bx}=2i\sum_{b\in B}\sin(2\pi bx).
 \end{eqnarray*} 
Hence 
$$\max_{x\in[0, \frac13]}\left|\sum_{b\in B}\sin(2\pi bx)\right|\geq\left(\frac1{42}\right)^{2M-1}\cdot(\# B).$$
\eproof

\medskip

     Indeed,  using Proposition \ref{exam-sine-conjecture}(i) and following the same  argument  of (i) implies (ii) in the proof of Theorem \ref{maintheorem3}, we can also give another proof for Theorem \ref{lemma_bounded_B} with a subsequence such that $\#B_n$ is bounded.

 

\subsection{Finite set of integers that is not so sparse} For the  finite set $B\subset{\mathbb N}$, the following two results indicate that   if $B$  is not so sparse inside $[0, \max B]$, then $\sum_{b\in B}\sin 2\pi bx$ has an uniform lower bound on $[0, \frac12-\delta_0]$.
 
 \medskip
 
\noindent  \textbf{Proof of Proposition \ref{exam-sine-conjecture} (ii):}
For any $B\in {\mathcal C}_c$, denote $p_B=\max B$ and $M_B=\#B$, then by the assumption, 
$$\frac{M_B}{p_B}\geq c >0.$$
Choose $x_0=\frac{1}{4p_B}$ which is a point in $[0, \frac14]$. For any $b\in B$, we have $2\pi bx_0\in[0, \frac{\pi}{2}]$. Since $\sin x$ is increasing on $[0, \frac{\pi}{2}]$ and  $\frac{2}{\pi}x\le\sin x\le x$ on $[0, \frac{\pi}{2}]$, we have 
\begin{eqnarray*}
 \sum_{b\in B}\sin(2\pi bx_0) &\geq&\sum_{b=0}^{M_B-1}\sin(2\pi bx_0)\\
  &=&\frac{\sin (M_B\pi x_0)\sin ((M_B-1)\pi x_0)}{\sin\pi x_0}\\
 &\geq&\frac{2M_B x_0\cdot 2(M_B-1)x_0}{\pi x_0}\\
 &\geq&\frac{M_B^2}{2\pi p_B}\geq\frac{c}{2\pi}\cdot  \#B.
\end{eqnarray*}
We complete the proof.
\eproof

\medskip
As a simple example for Proposition \ref{exam-sine-conjecture}(ii), we can take $B = \{0,1,...,p_B-1\}$. Then $\#B/p_B = 1$ and thus $B\in {\mathcal C}_1$.  This proposition can also be slightly  generalized in the following form.

\begin{lemma}
 For any  $B\in{\mathcal C}_{c, \ell} : = \left\{B\subset{\mathbb N}: \frac{\#(B\cap[\frac{p_B}{2^{\ell}}, p_B])}{\# B}\geq c \right\}$,
 where $p_B = \max B$ and $c, \ell>0$ are some fixed constants, we have
$$
\max_{x\in[0, \frac14]}\left|\sum_{b\in B}\sin (2\pi b x)\right|\geq c\sin\left(\frac{\pi}{2^{1+\ell}}\right)\cdot(\# B).
$$
\end{lemma}
\begin{proof}
 The proof is similar to that of Proposition \ref{exam-sine-conjecture} (ii). It is clear that for any $b\in B$, we have $2\pi bx_0\in[0, \frac{\pi}{2}]$ with $x_0=\frac{1}{4p_B}$.  Since $\sin x$ is increasing on $[0, \frac{\pi}{2}]$, we have 
\begin{eqnarray*}
 \sum_{b\in B}\sin(2\pi bx_0) &\geq&\sum_{b\in B\cap[\frac{p_B}{2^{\ell}}, p_B]}\sin(2\pi bx_0)\\
 &\geq&\sin \left(\frac{2\pi x_0p_B}{2^{\ell}}\right)\cdot\#\left(B\cap[\frac{p_B}{2^{\ell}}, p_B]\right)\\
 &\geq&c\sin\left(\frac{\pi}{2^{1+\ell}}\right)\cdot\left(\# B\right).
\end{eqnarray*}
\end{proof}

\subsection{Lacunary set of integers.} Recall that a finite set of positive integers $B = \{b_1<b_2<...<b_M\}$ is called {\it $A$-lacunary}, where $A>1$, if $b_{j+1}\ge A  b_j$. The collection of all finite sets of $A$-lacunary integers will be denoted by ${\mathcal L}_A$, as in Proposition \ref{exam-sine-conjecture}(iii). We now give a proof of this part of the  proposition.

\medskip

\noindent\textbf{Proof of Proposition \ref{exam-sine-conjecture} (iii):}
Since $A>2$, there exists $\epsilon_0>0$ such that $\frac12-\frac{\sin^{-1}\epsilon_0}{\pi}
\geq\frac1{A}.$ Suppose  $B=\{b_n\}_{n=0}^M\in{\mathcal L}_A$, then 
\begin{equation}\label{eq-lacunary}
b_{n+1}\geq A b_n\geq\left(\frac12-\frac{\sin^{-1}\epsilon_0}{\pi}\right)^{-1}b_n, \quad 0\le n\le M-1.
\end{equation}
We claim that there is a sub-interval $I\subset[0, \frac12-\frac{\sin^{-1}\epsilon_0}{2\pi}]$ such that 
$$\sin2\pi b x\geq \epsilon_0,\quad \forall\  b\in B\setminus\{0\},\  x\in I.$$
With the claim proved, we will have $\sum_{b\in B}\sin (2\pi bx)\ge \epsilon_0 (\#B)$ for all $x\in I$ and hence the result follows.

\medskip

We now justify the claim by  finding  $I$ by inductively. Note that 
\begin{eqnarray*}
\sin2\pi bx\geq\epsilon_0  \Longleftrightarrow  x\in\left[\frac{\sin^{-1}\epsilon_0}{2\pi b}, \frac{1}{2b}-\frac{\sin^{-1}\epsilon_0}{2
 \pi b}\right]+\frac1b{\mathbb Z}. 
\end{eqnarray*}  
Suppose $b_0=0$ and $b_n>0$ for $1\le n\le M$.  Take the interval $$I_1=\left[\frac{\sin^{-1}\epsilon_0}{2\pi b_1}, \frac{1}{2b_1}-\frac{\sin^{-1}\epsilon_0}{2\pi b_1}\right].$$
 Then $\sin2\pi b_1 x\geq \epsilon_0$ for all $x\in I_1$.
From \eqref{eq-lacunary}, the length of $I_1$ satisfies $$|I_1|=\left(\frac1{2}-\frac{\sin^{-1}\epsilon_0}{\pi}\right)\cdot\frac1{b_1}\geq\frac{1}{b_2}.$$
As $\sin 2\pi b_2 x$ is $\frac{1}{b_2}-$period,  we can pick an interval $I_2$ such that
$$I_2\subset I_1\cap\left(\left[\frac{\sin^{-1}\epsilon_0}{2\pi b_2}, \frac{1}{2b_2}-\frac{\sin^{-1}\epsilon_0}{2
 \pi b_2}\right]+\frac1{b_2}{\mathbb Z}\right) \text{ and } |I_2|=\left(\frac1{2}-\frac{\sin^{-1}\epsilon_0}{\pi}\right)\cdot\frac1{b_2}.$$
Then for all $x\in I_2$, we have $\sin 2\pi b_1x\geq \epsilon_0 \text{ and } \sin 2\pi b_2x\geq \epsilon_0.$
 Suppose that $I_{M-1}$ has been found, which satisfies that $|I_{M-1}|=\left(\frac1{2}-\frac{\sin^{-1}\epsilon_0}{\pi}\right)\cdot\frac1{b_{M-1}}$ and 
$$\sin 2\pi b_n x\geq \epsilon_0, \quad \forall\ 1\le n\le M-1, x\in I_{M-1}.$$
Again by \eqref{eq-lacunary}, $|I_{M-1}|\geq \frac1{b_M}$ and also by $\sin 2\pi b_M x$ is $\frac{1}{b_M}-$period, we can select $I_M$ such that 
$$I_M\subset I_{M-1}\cap\left[\frac{\sin^{-1}\epsilon_0}{2\pi b_M}, \frac{1}{2b_M}-\frac{\sin^{-1}\epsilon_0}{2
 \pi b_M}\right]+\frac1{b_M}{\mathbb Z} \text{ and } |I_M|=\left(\frac1{2}-\frac{\sin^{-1}\epsilon_0}{\pi}\right)\cdot\frac1{b_M}.$$
Then, we have 
$$\sin 2\pi b_nx\geq \epsilon_0, \quad \forall\ 1\le n\le M, x\in I_M.$$ 
Take $I = I_M$ and our result follows.
\eproof

 \section{Remark and open questions}
 
\subsection{Cantor-Moran measures on $[0,1]$.} For the {\bf (Qu 1)} in the introduction with the Cantor-Moran measure supported on $[0,1]$, what is still unsolved is the case that $\{\nu_{>n}\}$ converges weakly to $\rho$ and $\lim_{n\rightarrow\infty}\#B_n = \infty$. However, the Bourgain's sum of sine theorem tells us that there exists Cantor-Moran measure generated by $\{(N_n,B_n)\}$  for which  an equi-positive subsequence cannot be constructed.  However, it may still be possible that equi-positive subsequence is constructible for those $B_n$ which generate Hadamard triples. Essentially, it is to solve the following question.

\medskip

{\bf (Qu 2):} Let 
$$
{\mathcal H}: = \{B\subset{\mathbb Z}: \exists N \ \mbox{such that}  \  B\subset\{0,1,...,N-1\} \ \mbox{and} \ (N,B,L)  \ \mbox{forms a Hadamard triple for some } \ L\}
$$
Does the sum of sine problem hold for ${\mathcal H}$? 

\medskip

It appears that the Hadamard triple imposes a rigid condition on the set $B$ we can choose, so ({\bf Qu 2}) may still have a positive answer. Nonetheless, we are not able to utilize the Hadamard triple assumption easily at this moment. 

\medskip

Our work also shows that if we replace Hadamard triple tower to a frame triple tower, then except the extreme case that $\{\nu_{>n}\}$ converges weakly to $\rho$ and $\lim_{n\rightarrow\infty}\#B_n = \infty$,  we can construct a Fourier frame for the Cantor-Moran measure. Frame triple condition is now a more flexible condition than the Hadamard triple. The following question is now naturally raised.

\medskip

{\bf (Qu 3):} Given a Cantor-Moran measure $\mu$ generated by $\{(N_n,B_n)\}$. Can we always find a factorization of the Cantor-Moran measure so that we can construct a frame triple tower? 

\medskip

Under the special case that $(N_n,B_n) = (3,\{0,2\})$, this is exactly the question whether the middle-third Cantor measure admit a Fourier frame. In this case, $\nu_{>n}$ is always the middle-third Cantor measure, so we will be able to construct a Fourier frame immiediately once we can construct a frame triple tower upon some factorization.

\medskip

 A way to construct a frame triple tower is to use the Kadison-Singer theorem \cite{MSS} (see also \cite{DEL2018} for some recent advance). Another possible approach is to construct a continuous frame for the middle-third Cantor measure. Then with the Kadison-Singer theorem, Freeman and Speegle \cite{FS2018} is able to sample a discrete Fourier frame out of the continuous frame. Constructing a continuous frame for the middle-third Cantor measure was studied in \cite{DHW2014}, however, from what they tried, there appears to be no direct and natural way to find a continuous frame for the Cantor measure.

\medskip

\subsection{Cantor-Moran measures outside $[0,1]$} Our method can also be used to study Cantor-Moran measure that is not necessarily supported inside $[0,1]$. In other words, some $B_n$ is not a subset of $\{0,1,...,N_n-1\}$. In this case, we will require a {\it no-overlap condition}. Given $\{(N_n,B_n)\}$ that generates a Cantor-Moran measure $\mu$, we write $\mu = \mu_n\ast\mu_{>n}$. Denote by $K_n$, $K_{>n}$ to be the support of $\mu_n$ and $\mu_{>n}$ respectively. Note that $K_n$ consists only of finitely many points.  We say that $\mu$ satisifies the {\it no-overlap condition} if 
$$
\mu (({\bf b}+K_{>n})\cap ({\bf b}'+K_{>n})) = 0
$$
 for all ${\bf b}\ne {\bf b}'\in K_n$. When $B_{n}\subset\{0,1,...,N_{n-1}\}$ for all $n$, the no-overlap condition is easily satisfied since the intersection is either empty or consists only of one point. In \cite[Theorem 3.3]{DEL2018}, it was proved that under the no-overlap condition, Theorem \ref{theorem_LW} continues to hold. 
 
 \medskip

The following theorem is still true by some obvious modification of our results.
 
 \begin{theorem}
 Let $\{(N_n,B_n,L_n)\}$ be a frame triple tower (or respectively a Hadamard triple tower).  Suppose that $\mu(N_n,B_n)$ satisfies the no-overlap condition. If 
 \begin{enumerate}
     \item there exists a compact set $K$ such that all support of $\nu_{>n}$ are in $K$ and,
     \item there exists a subsequence $\{n_k\}$ such that  $\{\nu_{>n_k}\}$ is an admissible sequence in the sense that their periodic zero set are empty, so is its weak limit.
 \end{enumerate}
 Then $\mu(N_n,B_n)$ is a frame-spectral measure (or respectively a spectral measure). 
 \end{theorem}
 
 \begin{proof}
Note that all theorems in Section 3 do not require the assumption that $\nu_{>n}$ is supported in $[0,1]$.  The first condition ensures that equicontinuity of $\{\nu_{>n}\}$ still holds, and the second condition ensures the equi-positivity holds by Theorem \ref{Key_proposition}. With equi-positivity, Theorem \ref{theorem_admissible_spectral} and Theorem \ref{theorem_LW} (modified under the no-overlap condition) gives us our desired conclusion.
 \end{proof}

 \medskip
 
 As the measure is now outside $[0,1]$, the pull-back measures of $\nu_{>n}$ to ${\mathbb T}$ becomes much more complicated. Therefore, condition (ii) in the previous theorem cannot be easily analyzed like Section 4. The following example is previously known \cite{AHL2015} to be the counterexample of {\bf (Qu 1)}. Let us also analyze its admissibility and equi-positivity.

 \begin{example}
 Consider 
 $$
 \mu = \delta_{\frac12\{0,1\}}\ast\delta_{\frac{1}{2^2}\{0,3\}}\ast\delta_{\frac{1}{2^3}\{0,3\}}\ast...
 $$
with all $N_n = 2$ and $B_n = \{0,3\}$ except $n=1$ which equals to $\{0,1\}$. From \cite{AHL2015}, it is known that $\mu$ is not a spectral measure even though $(N_n,B_n,L_n)$ are Hadamard triples with $L_n = \{0,1\}$.  

\medskip

In this example, $\nu_{>n} = \frac13{\mathcal L}|_{[0,3]}$ for all $n>1$, the normalized Lebesgue measure supported on [0,3]. Hence, $\widehat{\nu_{>n}}$ has a non-empty periodic zero set (e.g. $1/3\in{\mathcal Z}(\nu_{>n})$). Hence, $\{\nu_{>n}\}$ cannot be admissible. Clearly, it cannot be equi-positive either since there is no way to move $1/3$ by an integer to make the Fourier transform be non-zero. 
 \end{example}

We see that in some sense the gcd$(B_n)>1$ for infinitely many $n$ may create some trouble. Assuming all gcd$(B_n)=1$ is a good starting place to study the general case.  The following conjecture below  may be a reasonable conjecture to this end.

\begin{conjecture}
 Suppose that gcd$(B_n)=1$ for all $n$ and suppose that $\{(N_n,B_n,L_n)\}$ forms a Hadamard triple tower with $\mu(N_n,B_n)$ having a compact support. Then $\mu(N_n,B_n)$ always has the no-overlap condition and it is a spectral measure if and only if $\{\nu_{>n}\}$ has an equi-positive subsequence. 
\end{conjecture}

\noindent {\bf Acknowledgement.} The authors would like to thank Professors Tamas Erdelyi and Mihalis Kolountzakis for their insightful discussions through E-mails. Special thanks to Mihalis Kolountzakis for pointing out  Bourgain's paper \cite{B1984} which helps us to complete the project.


\begin{thebibliography}{9999}
 
 \bibitem{AH2014}
 L.~X. An, X.~G. He, A class of spectral Moran measures, J. Funct. Anal. 266 (2014), no. 1, 343--354.
 
 \bibitem{AHL2015}
 L.~X. An, X.~G. He, K.~S. Lau, Spectrality of a class of infinite convolutions, Adv. Math. 283 (2015), 362--376.
 
 \bibitem{AHH2018}
 L.~X. An, X.~G. He, L. He, Spectrality and non-spectrality of the Riesz product measures with three elements in digit sets, to appear in J. Funct. Anal.
 
 \bibitem{B1984}
 J. Bourgain, Sur les sommes de sinus, Sem. Anal. Harm., Publ. Math. d'Orsay 
(1983) 83--101.

 \bibitem{B1986}
J. Bourgain, Sur le minimum d'une somme de cosinus, Acta Arith. 45 (1986), 381-389.
 
 \bibitem{BC2003}
 G. Benke, D.-C, Chang, On the sum of sine products, J. Math. Anal. Appl. 284 (2003), 647  --655.
 
 \bibitem{C1965}
 S. Chowla, Some applications of a method of A. Selberg, J. Reine Angew. Math. 217 (1965), 128--132.
 
\bibitem{C2010} 
K.~L. Chung, A course in probability theory, Stanford University Press, 2010. 
 
 \bibitem{D2012}
X.~ R. Dai, When does a Bernoulli convolution admit a spectrum? Adv. Math. 231 (2012) 187--208.

\bibitem{DHL2013}
X. R. Dai, X. G. He, C. -K. Lai, Spectral property of Cantor measures with consecutive digits, Adv.
Math. 242 (2013) 1681--1693.

\bibitem{DHL2014}
X. R. Dai, X.G. He, K.-S. Lau, On spectral N-Bernoulli meausure, Adv. Math. 259 (2014) 511--531.

\bibitem{DS2015}
 X. R. Dai,   Q.Y. Sun, Spectral measures with arbitrary Hausdorff dimensions. J. Funct. Anal.  268  (2015),  no. 8, 2464–-2477.
 
 
\bibitem{DHS2014} D.~E. Dutkay, D.G. Han,  Q.Y. Sun, Divergence of the mock and scrambled Fourier  series on fractal measures, Trans, Amer. Math. Soc. 366 (2014), 2191--2208.

\bibitem{DHW2014} 
D. Dutkay, D.-G. Han,  E. Weber, Continuous and discrete Fourier frames for Fractal measures. Trans. Amer. Math. Soc. 366 (2014), no. 3, 1213-1235.

\bibitem{DEL2018}
D.~E. Dutkay, S. Emami and C.-K Lai, Existence and exactness of exponential Riesz sequences and frames for fractal measures, https://arxiv.org/abs/1809.06541, 2018.



\bibitem{DH2016}
D.~E. Dutkay, J.~Haussermann, Number theory problems from the harmonic analysis of a fractal, J. Number Theory 159 (2016), 7--26.

 \bibitem{DHL2017}
D.~E. Dutkay, J.~Haussermann,  C.~K. Lai, Hadamard triples generate self-affine spectral measures, Trans. Amer. Math. Soc., 371 (2019), 1439-1481. 

\bibitem{DL2017}
D.~E. Dutkay,  C.~K. Lai, Spectral measures generated by arbitrary and random convolutions, J. Math. Pures Appl. (9) 107 (2017), no. 2, 183-204.

\bibitem{DK2018}
  D.~E. Dutkay, I. Kraus, Number theoretic considerations related to the scaling of spectra of Cantor-type measures, Anal. Math. 44 (2018), no. 3, 335--367.
  
  \bibitem{FS2018}
   D. Freeman and D. Speegle, The discretiztion problem for continuous frames, Adv. Math,  (2018), in press.
  
\bibitem{FHW2018}
Y.~S. Fu, X.~G. He, Z.~Y. Wen, Spectra of Bernoulli convolutions and random convolutions, J. Math. Pures Appl. (9) 116 (2018), 105--131.


\bibitem{F1974}
B.~Fuglede, Commuting self-adjoint partial differential operators and a group
  theoretic problem,  J. Funct. Anal., (1974) 16(1):101--121.

\bibitem{GL2014}
J.~P. Gabardo, C.~K. Lai, Spectral measures associated with the factorization of the Lebesgue measure on a set via convolution,  J. Fourier Anal. Appl. 20 (2014), no. 3, 453--475.

\bibitem{HH2017}
L. He, X.~G. He, On the Fourier orthonormal bases of Cantor-Moran measures, J. Funct. Anal. 272 (2017), no. 5, 1980--2004.

\bibitem{JP1998}
P. Jorgensen, S. Pedersen, Dense analytic subspaces in fractal $L^2$ spaces, J. Anal. Math. 75 (1998) 185-228.

\bibitem{K1985}
J.~P. Kahane, Some random series of functions,  Cambridge University Press, 2nd ed., 1985. 

\bibitem{K2004}
Y.~ Katznelson, An introduction to harmonic analysis, Cambridge University Press, 2004.

\bibitem{K1994-1}
M.N. Kolountzakis, On nonnegative cosine polynomials with nonnegative, integral coefficients, Proc. Amer. Math. Soc. 120 (1994), vol. 1, 157--163.

\bibitem{K1994-2}
M.N. Kolountzakis, A construction related to the cosine problem, Proc. Amer.Math. Soc. 122 (1994), vol.
4, 1115–-1119.

\bibitem{K1994-3}
M. N. Kolountzakis, Probabilistic and constructive methods in harmonic analysis and additive number theory, PhD Dissertation, Stanford 1994.

\bibitem{KM1}
M. N. Kolountzakis, M.,Matolcsi, Complex Hadamard matrices and the spectral set conjecture, Collec. Math. Vol. Extra (2006) 281--291.

\bibitem{KM2}
M. N. Kolountzakis, M.,Matolcsi, Tiles with no spectra, Forum Math 18 (2006) 519--528.

\bibitem{LW2002}
I. $\L$aba, Y. Wang, On spectral Cantor measures, J. Funct. Anal. 193 (2002) 409--420.

\bibitem{LW2017}
C.~K. Lai, Y. Wang, Non-spectral fractal measures with Fourier frames, J. Fractal Geom. 4 (2017), no. 3, 305--327.

\bibitem{MSS}
A. W. Marcus, D. A. Spielman, and N. Srivastava, Interlacing families II: Mixed characteristic polynomials and the Kadison-Singer problem, Ann. of Math.  182 (2015), 327-350.

\bibitem{Moran}
P.A. Moran, Additive functions of intervals and Hausdorff measure. Proc. Cambridge Philos. Soc.  42 (1946). 15-23.

\bibitem{N94}
F. Nazarov, Local estimates for exponential polynomials and their applications to inequalities of the uncertainty principle type, St. Petersburg Math. J. 5 (1994), 663 --717.



\bibitem{S2000}
R.~S. Strichartz, Mock Fourier series and transforms associated with certain Cantor measures.
J. Anal. Math. 81 (2000) 209-238.

\bibitem{S2006} R.~S. Strichartz, Convergence of mock Fourier series, J. Anal. Math. (2006) 99, 333-353. 

\bibitem{Shi2018}
R. Shi, Spectrality of a class of Cantor–Moran measures, to appear in J. Funct. Anal.

\bibitem{TF2018}
M.-W. Tang, F.-L Yin, Spectrality of Moran measures with four-element digit sets, J. Math. Anal.  Appl. 461 (2018), 354-363.

\bibitem{T2004}
T. Tao, Fuglede’s conjecture is false in 5 and higher dimensions, Math. Res. Lett. 11 (2004) 251--258.

\end{thebibliography}
\end{document}